\documentclass[11pt,a4paper]{amsart}

\usepackage{amssymb}
\usepackage{amsmath,graphics,verbatim}
\usepackage{latexsym}
\usepackage{eucal}
\usepackage{a4wide}

\newtheorem{proposition}{Proposition}[section]
\newtheorem{theorem}{Theorem}[section]
\newtheorem{lemma}[theorem]{Lemma}
\newtheorem{prop}[theorem]{Proposition}
\newtheorem{coro}[theorem]{Corollary}
\theoremstyle{remark}
\newtheorem{remark}[theorem]{Remark}

\newcommand{\mc}{\mathcal}

\newcommand{\rr}{\mathbb{R}}

\newcommand{\nn}{\mathbb{N}}
\newcommand{\cc}{\mathbb{C}}

\newcommand{\zz}{\mathbb{Z}}

\newcommand{\eps}{\epsilon}

\newcommand{\pl}{\partial}
\newcommand{\x}{\times}

\newcommand{\til}{\widetilde}
\newcommand{\bbar}{\overline}

\newcommand{\supp}{\textrm{supp}}
\newcommand{\cjd}{\rangle}
\newcommand{\cjg}{\langle}

\newcommand\RR{\mathbb{R}}
\newcommand\CC{\mathbb{C}}

\newcommand{\SC}{\ensuremath{\mathrm{sc}}}

\renewcommand\Re{\operatorname{Re}}
\renewcommand\Im{\operatorname{Im}}

\newcommand\MMb{M^2_b}
\newcommand\WF{\operatorname{WF}}

\newcommand\mf{\operatorname{mf}}

\newcommand\la{\lambda}
\newcommand\demi{\frac{1}{2}}
\newcommand\ndemi{\frac{n}{2}}
\newcommand\Id{\operatorname{Id}}
\newcommand{\zf}{\mathrm{zf}}
\newcommand{\bfo}{\mathrm{bf}_0}
\newcommand{\rbo}{\mathrm{rb}_0}
\newcommand{\lbo}{\mathrm{lb}_0}
\newcommand{\lb}{\mathrm{lb}}
\newcommand{\rb}{\mathrm{rb}}

\newcommand{\bfc}{\mathrm{bf}}
\newcommand{\sca}{\mathrm{sc}}

\newcommand{\ilabel}{\label}

\begin{document}
\title[Uniform Sobolev estimates for non-trapping metrics]{Uniform Sobolev estimates for non-trapping metrics}

\author{Colin Guillarmou}
\address{DMA, U.M.R. 8553 CNRS\\
Ecole Normale Sup\'erieure\\
45 rue d'Ulm\\ 
F 75230 Paris cedex 05 \\France}
\email{cguillar@dma.ens.fr}

\author{Andrew Hassell}
\address{Department of Mathematics, Australian National University \\ Canberra ACT 0200 \\ AUSTRALIA}
\email{Andrew.Hassell@anu.edu.au}

%\author{Adam Sikora}

\begin{abstract}
We prove uniform Sobolev estimates $||u||_{L^{p'}} \leq C ||(\Delta-\alpha)u||_{L^{p}}$ for $\alpha\in\cc$ and $p=2n/(n+2), p'=2n/(n-2)$
on non-trapping asymptotically conic manifolds of dimension $n$, generalizing to non-constant coefficient Laplacians a result of 
Kenig-Ruiz-Sogge \cite{KRS}.   
\end{abstract}

\maketitle

\section{Introduction}\ilabel{sec:intro}

In this paper, we consider a class of complete non-compact Riemannian manifolds of dimension $n$, 
which generalize the Euclidean structure near infinity in a natural way. These are \emph{asymptotically conic} manifolds, defined as follows:  
$(M,g)$ is asymptotically conic if $M$ is the interior of a smooth compact manifold with boundary $\bbar{M}$, $g$ is a smooth metric on $M$ 
such that there exists a smooth boundary defining function $x$ on $\bbar{M}$ with $(M,g)$ isometric outside a compact set to 
\begin{equation}  (0,\eps)_x\x \pl\bbar{M} ,\textrm{ with metric } g=\frac{dx^2}{x^4}+\frac{h_x}{x^2}
\ilabel{metric}\end{equation}
where $h_x$ is a smooth one-parameter family of metrics on $\pl\bbar{M}$. Euclidean space $\RR^n$ fits into this framework, with
$\pl\bbar{M}=S^{n-1}$, $x=1/|z|$ where $z\in \rr^n$, and $h_x$ the standard metric on $S^{n-1}$. More generally, if $h_x$ is independent of $x$ for small $x$, the metric is of the form $dr^2 + r^2 h$ in terms of $r = 1/x$ for large $r$, hence conic near infinity. 

The Laplacian $\Delta$ associated to $g$ has only continuous spectrum given by the half-line $[0,\infty)$. Let  $dE_{\sqrt{\Delta}}(\lambda)$
be the spectral measure of $\sqrt{\Delta}$, defined by $F(\sqrt{\Delta})=\int_{0}^\infty F(\la)dE_{\sqrt{\Delta}}(\la)$ for all bounded functions
$F$. In \cite{GHS2}, the authors with Adam Sikora proved that,  when $n\geq 3$,  the spectral measure $dE_{\sqrt{\Delta}}(\lambda)$ maps $L^q(M)$ to $L^{q'}(M)$ boundedly for all $q$ in the range $[1, 2(n+1)/(n+3)]$, generalizing the Tomas-Stein restriction estimate to this class of manifolds.  
When in addition $(M,g)$ is \emph{non-trapping}, meaning that every geodesic on $M$ reaches spatial infinity in 
forward and backward time, then the following estimate was also shown in \cite{GHS2}:
\begin{equation}
\big\| dE_{\sqrt{\Delta}}(\lambda) \big\|_{L^q(M) \to L^{q'}(M)} \leq C \lambda^{n(1/q - 1/q') - 1}, \text{ for all } \lambda > 0
\ilabel{spe}\end{equation}
with a constant $C$ independent of $\lambda$. In general if $(M,g)$ is trapping, the estimate still holds for $0 < \lambda < \lambda_0$ with a constant $C$ depending on $\lambda_0$.

Resolvent estimates between $L^q$ spaces are also of interest for a number of reasons. The most classical is the Hardy-Littlewood-Sobolev inequality in $\RR^n$, a special case of which is 
$$
\| \Delta^{-1} u \|_{L^r(\RR^n)} \leq C \| u \|_{L^q(\RR^n)}, \quad 
\frac1{q} = \frac1{r} + \frac2{n}. 
$$
This was generalized in the following way by Kenig-Ruiz-Sogge \cite{KRS}: suppose that $L$ is any constant coefficient first order differential operator on $\RR^n$. Then for the `Sobolev exponents' $p = 2n/(n+2)$, $p' = 2n/(n-2)$,  there is an inequality
\begin{equation}
\big\| (\Delta + L) u \big\|_{L^{p}(\RR^n)} \geq C \| u \|_{L^{p'}(\RR^n)}, \text{ for all } u \in W^{2,p}(\RR^n) 
\ilabel{KRS}\end{equation}
with a constant $C$ independent of $L$. (In fact, they were able to obtain such a uniform estimate  even when $\Delta$ was replaced by a homogeneous second order constant coefficient differential operator, not necessarily elliptic but with non-degenerate leading symbol.) In particular, when $L$ is a constant $-\alpha \in \CC$, they showed 
\begin{equation}
\big\| (\Delta - \alpha) u \big\|_{L^{p}(\RR^n)} \geq C \| u \|_{L^{p'}(\RR^n)}, \text{ for all } u \in W^{2,p}(M). 
\ilabel{KRS2}\end{equation}
Inequality \eqref{KRS} was then used to deduce Carleman inequalities, and then unique continuation theorems, on $\RR^n$.

Given \eqref{spe} and \eqref{KRS2}, and noting that $p = 2n/(n+2)$ is within the range of validity of \eqref{spe},  it is natural to predict that similar uniform estimates hold on asymptotically conic non-trapping manifolds. Indeed, this is the case. Our main result is 

\begin{theorem}\ilabel{mainth} 
Let $(M,g)$ be an asymptotically conic non-trapping manifold of 
dimension $n \geq 3$. Let $p = 2n/(n+2)$ and $p'=2n/(n-2)$. Then there is a constant $C>0$ such that for all $\alpha \in \CC$, we have 
\begin{equation}
 \| u \|_{L^{p'}(M)} \leq C \big\| (\Delta_g - \alpha) u \big\|_{L^p(M)}  , \text{ for all } u \in W^{2,p}(M). 
\ilabel{uniform-delta-ineq}\end{equation}
Equivalently, for all $f\in L^p(M)$ and  all $\alpha \in \CC$, 
\begin{equation}
\big\| (\Delta_g - \alpha)^{-1} f \big\|_{L^{p'}(M)} \leq C \| f \|_{L^p(M)}.
\ilabel{uniform-resolvent-ineq}\end{equation}
When $\alpha > 0$, the operator in \eqref{uniform-resolvent-ineq}  may be taken to be either the  incoming or outgoing resolvents, $(\Delta - (\alpha \pm i0))^{-1}$. If the metric is trapping, the same estimate holds true for all $\alpha$ such that $\Re \alpha\leq A$ for  any $A>0$, 
with $C$ depending only on $A$.
\end{theorem}

\begin{remark}
In fact we prove a more general result in the non-trapping case (see remarks \ref{2(n+1)/(n+3)}, \ref{endpointhigh} and \ref{endpointlow}): for all $q\in [\frac{2n}{n+2},\frac{2(n+1)}{n+3}]$, there is $C>0$ such that for all 
$\alpha\in\cc$  and all $f\in L^q(M)$
\begin{equation}
\big\| (\Delta_g - \alpha)^{-1} f \big\|_{L^{q'}(M)} \leq C|\alpha|^{n(\frac{1}{q}-\demi)-1} \| f \|_{L^q(M)}.
\ilabel{L^qL^q'}\end{equation}
If $\eta>0$ is small, the same estimate holds true for all $q\in [\frac{2n}{n+2},2]$ and $|\arg(\alpha)|>\eta>0$, with $C$ depending only on $\eta$: see \eqref{interp}.
\end{remark}
\begin{remark} 
It is not completely obvious that \eqref{uniform-delta-ineq} and \eqref{uniform-resolvent-ineq} are equivalent. To see that they are, 
first consider the case $\Im \alpha \neq 0$ or $\Re \alpha < 0$. Then $(\Delta - \alpha)^{-1}$ is a pseudodifferential operator of order $-2$ 
(more precisely in the scattering calculus \cite{Mel}). Therefore $\nabla \nabla (\Delta - \alpha)^{-1}$ is a pseudodifferential operator of order $0$, and hence bounded on $L^p(M)$, $1 < p < \infty$. Hence $(\Delta - \alpha)^{-1}$ is an isomorphism between $L^p(M)$ and $W^{2,p}(M)$, and using this we easily see that \eqref{uniform-delta-ineq} and \eqref{uniform-resolvent-ineq} are equivalent and $W^{2,p}(M)\subset L^{p'}(M)$.
In the case that $\alpha > 0$, we show that \eqref{uniform-resolvent-ineq} implies \eqref{uniform-delta-ineq}. Let $u \in C_c^\infty(M) \subset W^{2,p}(M)$ and $f=(\Delta-\alpha)u$. Then, by \eqref{uniform-resolvent-ineq}, Êone has for all $\alpha>0,\epsilon > 0$, 
\begin{equation}
\big\| (\Delta - (\alpha + i\epsilon))^{-1} f \big\|_{L^{p'}} \leq C \| f \|_{L^p}.
\ilabel{distlimit}\end{equation}
Therefore there is a sequence of $\epsilon$ tending to zero such that $(\Delta - (\alpha + i\epsilon))^{-1} f $ has a weak limit. By definition,  $(\Delta - (\alpha + i0))^{-1} f $ is the distributional  limit of $(\Delta - (\alpha + i\epsilon))^{-1} f $ as $\epsilon $ tends to zero. 
But $(\Delta - (\alpha + i\epsilon))^{-1} f=u+i\eps(\Delta - (\alpha + i\epsilon))^{-1}u$ and $u \in x^{1/2 + \epsilon} L^2(M)$, hence by the limiting absorption principle \cite{Mel}, $\lim_{\epsilon \to 0}(\Delta - (\alpha + i\epsilon))^{-1}u$ exists and thus $(\Delta - (\alpha + i0))^{-1} f=u$. 
By $\eqref{distlimit}$, this limit lies in $L^{p'}$ and satisfies 
$$
%\big\| (\Delta - (\alpha + i0))^{-1} f \big\|_{L^{p'}} \leq C \| f \|_{L^p} \text{ or equivalently }  
\| u \|_{L^{p'}(M)} \leq C \big\| (\Delta_g - \alpha) u \big\|_{L^p(M)} .
$$
A density argument then shows this for all $u \in W^{2,p}(M)$. 
The converse implication is obtained through similar reasoning. 
\end{remark}

\begin{remark} The example in Remark 8.8 of \cite{GHS2} (the connected sum of $\RR^n$ and a round sphere) shows that, in the absence of any nontrapping assumption, the $L^p \to L^{p'}$ norm of the resolvent $(\Delta - \alpha \pm i0)^{-1}$ may grow exponentially as $\alpha \to \infty$.
\end{remark}

\subsection{Outline of the proof and previous results}

We prove Theorem~\ref{mainth} in two steps, first in the region $\{{\rm Re}(\alpha)\leq 0\textrm{ or } |\arg \alpha| < \eta\}$ 
for any $\eta>0$ and then in $\{|\arg \alpha| < \eta\}$. 

The estimate in the first region is a consequence of ellipticity 
and the Sobolev estimate $||\Delta u||_{L^p}\geq C||u||_{L^{p'}}$, which follows 
for instance from Gaussian estimates on the heat kernel $e^{-t\Delta}(z,z')$ and \eqref{spe} for $\la$ near $0$ (an alternative approach
is to use the inverse $\Delta^{-1}$ of $\Delta$ constructed in \cite{GH} --- see Remark~\ref{rem:Green}). We do this in Sections~\ref{sec:leq0} and \ref{sec:outside-sector}. 
We also show that \eqref{spe} implies a uniform estimate on the difference $(\Delta - \alpha)^{-1} - (\Delta - \overline\alpha)^{-1}$ between the resolvent and its formal adjoint, for all $\alpha \in \CC \setminus \RR$.

To prove the estimate within the sector $|\arg \alpha| < \eta$, i.e.\  close to the spectrum, we need more detailed properties of the spectral measure. Using complex interpolation, we show that if we had pointwise estimates of the form 
$$
\Big| \pl^\ell_{\la}dE_{\sqrt{\Delta}}(\la) (z,z') \Big|\leq C\la^{n-1-\ell}(1+\la d(z,z'))^{-\frac{n-1}{2}+\ell}, \quad \forall\, \ell\leq n/2
$$
where $d(z,z')$ is the Riemannian distance between $z, z' \in M$, then \eqref{uniform-resolvent-ineq} would be a consequence.
However, these estimates, which hold for the Euclidean Laplacian, do not hold in general for asymptotically conic 
non-trapping manifolds, essentially because there can be conjugate points for the geodesic flow. 
As in \cite{GHS2}, the way we bypass this problem is through a microlocal partition of the identity, 
$\Id = \sum_{i=1}^N Q_i(\alpha )$, where $Q_i(\alpha)$ are pseudodifferential operators depending on $\alpha$. If $i,j$ are such that 
the microsupports (which will be also called wavefront sets) of $Q_i(\alpha)$ and $Q_j(\alpha)$ are chosen appropriately, 
we have from  \cite{GHS2}  that there is $C>0$ such that for $\lambda/\alpha\in[1-\delta,1+\delta]$
\begin{equation}\ilabel{kernelest}
\big|Q_i(\alpha)\pl^\ell_{\la}dE_{\sqrt{\Delta}}(\la)Q_j(\alpha) (z,z')|\leq C\la^{n-1-\ell}(1+\la d(z,z'))^{-\frac{n-1}{2}+\ell}.
\end{equation}
The condition on $i,j$ is essentially that the microsupports of $Q_i$ and $Q_j$ are sufficiently close in phase space (in a sense that will be explained later). This allows us to prove `near-diagonal' estimates on the resolvent, that is, for $Q_i(\alpha) (\Delta - (\lambda \pm i\gamma)^2)^{-1} Q_j(\alpha)$ when $i,j$ are as above. 

It remains to discuss the `off-diagonal' estimates, that is, for $Q_i(\alpha) (\Delta - (\lambda \pm i\gamma)^2)^{-1} Q_j(\alpha)$ when $Q_i(\alpha),Q_j(\alpha)$ have  separated microsupports. 
Our proof of these estimates uses properties of the resolvent  proved in \cite{HW, GHS1}, namely that it is the sum of a pseudodifferential operator and  a `Legendre distribution'. The Legendre part is oscillatory at the boundary and can be understood as a sort of Fourier Integral operator,  and has a well-defined microlocal support, analogous to the canonical relation of an FIO; for the outgoing/incoming resolvent, this microlocal support is essentially the forward/backward geodesic flow relation on the cotangent bundle of $M \times M$. Because of this oscillatory structure, we can understand the composition $Q_i (\Delta - (\alpha \pm i0))^{-1} Q_j$ microlocally --- see 
 Proposition~\ref{WFprop} (which is taken from \cite[Section 7]{GHS2}).  
 
Our strategy is then to choose the partition $Q_i$ so that, either the microsupports of $Q_i$ and $Q_j$ are very close, in which case we are effectively in the `on-diagonal' case, where \eqref{kernelest} holds, or $Q_i$ has the property of either being `not outgoing-related' or `not incoming-related' to $Q_j$ in the sense explained in section \ref{psido}. 
 Heuristically  `not outgoing-related' means that there is no point in the microsupport of $Q_i$ that is obtained from the microsupport of $Q_j$ by forward geodesic flow. We give a construction of such a partition in Section~\ref{micpart}.
 We show from this that the \emph{outgoing} resolvent, $R(\lambda + i \gamma)$ for $\gamma > 0$, is essentially trivial when sandwiched between $Q_i$ and $Q_j$ if $Q_i$ is not outgoing-related to $Q_j$. 
 If $Q_i$ is not incoming-related to $Q_j$ then the same is true for the \emph{incoming} resolvent (that is, the case $\gamma < 0$). 
Thus, for each pair $(Q_i, Q_j)$, we get off-diagonal estimates for at least one of the incoming or outgoing resolvents. Since we already observed that we have an $L^p \to L^{p'}$ estimate for the \emph{difference} of the resolvents, this completes the proof.\\

In Appendix \ref{sec:alternative}, we also give a second, alternative proof of the off-diagonal resolvent estimates Êbased on positive commutator,  which does not rely on the Legendre structure of the resolvent. 
%We point out, however, that all the information about the Legendre structure of the spectral measure needed in this paper is contained in Lemma~\ref{QEQ}, which can be used as a `black box' for the current purposes. 
Since we think this could useful in other situations, the positive commutator proof 
provides some statement similar to propagation of singularities (ie. semiclassical wave-front set, including wave-front set `at infinity') for the incoming and outgoing resolvents, and we wrote it without assuming non-trapping geometry but rather 
certain assumptions like polynomial resolvent estimates on the spectrum, see Lemma \ref{lem:poscomm2}.\\

The first uniform Sobolev estimate of this type \eqref{uniform-delta-ineq}Ê appeared in the work of Kenig-Ruiz-Sogge \cite{KRS} for homogeneous second order operators on $\rr^n$, and it was  used to prove $L^p$ Carleman estimates and unique continuation. Shen \cite{Sh} proved that for the torus $\mathbb{T}^n=\rr^n/\zz^n$, 
\begin{equation}\label{dksest} 
||(\Delta-\alpha)^{-1}||_{L^p\to L^{p'}}\leq C \quad \text{ in the region } \, ({\rm Im}\, \alpha)^2>\delta \, {\rm Re}(\alpha), \quad |\alpha|>\delta,
\end{equation}
where $C$ depends only on $\delta>0$. This estimate was generalized by Dos Santos Ferreira-Kenig-Salo \cite{DKS} to compact manifolds
of dimension $n\geq 3$ and very recently, Bourgain-Shao-Sogge-Yao \cite{BSSY} proved that this estimate is sharp in general on compact manifolds since for the sphere $S^n$ (or Zoll manifolds), the regions where the estimate \eqref{dksest} can not be made smaller. They also show in this paper some equivalence between $L^p\to L^{p'}$ norms  of spectral projectors of $\Delta$ 
in frquency windows $[\la-\eps(\la),\la+\eps(\la)]$ with $\eps(\la)\to 0$ and $L^p\to L^{p'}$ estimates for $(\Delta-\la+i\gamma)^{-1}$ 
in $|\gamma|\geq \eps(\la)$ and $\la>1$.  

To compare these compact results with our case, we point out two main differences:  first the non-trapping condition which allows us to get  
estimates down to the spectrum, like in $\rr^n$, but brings technical difficulties coming from the complicated structure
of the resolvent near spatial infinity on asymptotically conic manifolds.
The second difference with the compact setting is that we deal with the behaviour 
at small frequencies, where we need to use the result of \cite{GHS1} microlocally analyzing 
the resolvent and spectral measure on the spectrum on the frequency window $[0,\lambda_0]$.\\

 \textbf{Acknowledgements.} We thank Adam Sikora and Jean-Marc Bouclet for useful discussions.
  C.G. thanks the Mathematics Department at ANU where part of this work was done,
and the PICS-CNRS Progress in Geometric Analysis and Applications between ANU and CNRS. C.G. is partially supported by 
 grants ANR-09-JCJC-0099-01 and  ANR-10-BLAN 0105. 
 A.H. acknowledges the support of the Australian Research Council through a Future Fellowship FT0990895 and  Discovery Grant DP1095448. 

%%%%%%%%%%%%%%%%%%%%%%%%%%%%%%
%%%%%%%%%%%%%%%%%%%%%%%%%%%%%%
%%%%%%%%%%%%%%%%%%%%%%%%%%%%%%

\section{Estimates when $\Re \alpha \leq 0$}\ilabel{sec:leq0}

We first consider the case $\alpha = 0$; the case $\Re \alpha\leq 0$ will then follow easily. 

\subsection{Sobolev estimate for $\Delta$}

We first recall the restriction estimate of \cite{GHS2}: let $(M,g)$ be an asymptotically conic manifold in dimension $n\geq 3$, let 
 $1\leq q\leq 2\frac{n+1}{n+3}$ and $q'$ its conjugate exponent, then for all $\la_0>0$, there exists $C>0$ such that for all $\la_0>\la>0$
\begin{equation}\ilabel{restrictionest}
||dE_{\sqrt{\Delta}}(\la)||_{L^q\to L^{q'}}\leq C\la^{n(\frac{1}{q}-\frac{1}{q'})-1}. 
\end{equation}

\begin{prop} Let $(M,g)$ be an asymptotically conic manifold of dimension $n\geq 3$.
We then have the classical Sobolev inequality: there exists $C>0$ such that for all $u \in C_c^\infty(M)$
\begin{equation}\ilabel{sobolev}
\cjg \Delta u,u\cjd \geq C||u||^2_{L^{p'}}, \quad p' = \frac{2n}{n-2}. 
\end{equation}
\end{prop}
\begin{proof} 
Varopoulos \cite{Var} proves that on any Riemannian manifold $(M,g)$, a bound of the form 
$||e^{-t\Delta}||_{L^1\to L^\infty}\leq Ct^{-n/2}$ for all $t>0$ implies  the Sobolev estimate \eqref{sobolev}.
This heat operator estimate is proved for $t\in(0,1)$ on any complete Riemannian  manifold by Cheng-Li-Yau \cite{ChLiYa}, it remains to prove it for large time. First, we write
\[e^{-t\Delta}1_{[0,1]}(\Delta)= \int_{0}^1e^{-t\lambda^2} dE_{\sqrt{\Delta}}(\la)d\la\]
and use \eqref{restrictionest} with $p=1$ to get directly
\[||e^{-t\Delta}1_{[0,1]}(\Delta)||_{L^1\to L^\infty}\leq Ct^{-\ndemi}.\]
By the estimate of \cite{ChLiYa}, we know that for any $0<\eps<1/4$, $e^{-\eps \Delta}$ is bounded as a map $L^1(M)\to L^2(M)$ and as a map $L^2\to L^\infty(M)$, thus we directly obtain for $t\geq 1$
\[  ||e^{-t\Delta}1_{[1,\infty)}(\Delta)||_{L^1\to L^\infty}\leq C||e^{-(t-2\eps)\Delta}1_{[1,\infty)}(\Delta)||_{L^2\to L^2}\leq Ce^{-t/2}.\]
\end{proof}

\subsection{Sobolev estimate for $\Re \alpha\leq 0$}

It is now easy to show
\begin{proposition} 
Suppose that $\Re \alpha \leq 0$. Then \eqref{uniform-delta-ineq} holds. 
\end{proposition}\label{sobolevRe<0}
\begin{proof}
Let $\alpha$ satisfy $\Re \alpha \leq 0$ and let $u\in C_0^\infty(M)$ be not identically zero. Then by \eqref{sobolev},Ê
there exists $C>0$ independent of $u$ such that 
\begin{equation}\ilabel{sobolev2} 
||(\Delta-\alpha)u||_{L^{p}(M)} \geq \frac{\big| \cjg(\Delta-\alpha)u,u\cjd \big|}{||u||_{L^{p'}(M)}}\geq \frac{\cjg \Delta u,u\cjd}{||u||_{L^{p'}(M)}}\geq  C||u||_{L^{p'}(M)} .
\end{equation}
which achieves the proof.
\end{proof}

\begin{remark}\ilabel{rem:Green} We could alternatively use results of \cite{GH} at zero energy, which shows that the Green function $\Delta_g^{-1}(z,z')$ is bounded by a constant times $d(z,z')^{-(n-2)}$. We can then use an abstract Hardy-Littlewood-Sobolev inequality from \cite{GG}, valid on metric measure spaces such that the measure of a ball of radius $\rho$ is comparable to $\rho^n$, that states  that the kernel $d(z,z')^{-(n-2)}$ maps $L^q$ to $L^r$ provided $1/q = 1/r + 2/n$, $1 < q < r < \infty$. 
\end{remark}

%%%%%%%%%%%%%%%%%%%%%%%%%%%%%%%%%
%%%%%%%%%%%%%%%%%%%%%%%%%%%%%%%%%
%%%%%%%%%%%%%%%%%%%%%%%%%%%%%%%%%
%%%%%%%%%%%%%%%%%%%%%%%%%%%%%%%%%

\section{Estimates when $\Re \alpha > 0$ and $|\arg \alpha| \geq \eta > 0$}\ilabel{sec:outside-sector}

We now turn to the more interesting case when $\Re \alpha \geq 0$. We shall prove, in this section, some estimates outside 
conic neighbourhood of the positive real axis. Below, 
$\phi(t)$ is a smooth function vanishing except on the interval $[1 - \delta, 1 + \delta]$ and equal to $1$ on the interval $[1 - \delta/2, 1 + \delta/2]$. Also, we write $\alpha = \beta + i\gamma$, $\beta, \gamma \in \RR$ for the real and imaginary parts of $\alpha$. 

\begin{proposition}\ilabel{outsidecone} 
Let $\alpha = \beta + i\gamma$ with $\beta > 0$ and $|\arg \alpha| \geq \eta > 0$. Then there is a constant $C$  depending only on $\eta$ such that 
\begin{equation}\ilabel{delta-alpha}
\big\| (\Delta - \alpha)^{-1} \big\|_{L^p(M) \to L^{p'}(M)} \leq C.
\end{equation}
Moreover, for $\phi$ as above, 
%Let $\phi$ be a smooth function vanishing except on the interval $[1 - \delta, 1 + \delta]$ and equal to $1$ 
%on the interval $[1 - \delta/2, 1 + \delta/2]$. Then 
there exists $C>0$ such that for all $\Re \alpha > 0$, 
\begin{equation} \ilabel{disttospecleq1}
\Big\| \big( 1- \phi(\Delta/|\alpha|) \big)(\Delta-\alpha)^{-1} \Big\|_{L^p\to L^{p'}}\leq C. \end{equation}
\end{proposition}
\begin{proof}  From \eqref{sobolev} and duality, we know that $(\Delta + \beta)^{-\demi}$ is bounded as an operator 
$L^2\to L^{p'}$ and $L^p\to L^2$, thus 
there exists $C>0$ independent of $\beta$ such that for all $u\in C_c^\infty(M)$ 
\[ ||(\Delta-\alpha)^{-1}u||_{L^{p'}}\leq C|| (\Delta + \beta)^{\demi}(\Delta-\alpha)^{-1}u||_{L^2}\leq 
C^2 \big\| (\Delta + \beta)(\Delta-\alpha)^{-1} \big\|_{L^2\to L^2} ||u||_{L^p}. \]
Now the $L^2\to L^2$ bound for the operator in the right hand side is 
\[ \big\| (\Delta + \beta)(\Delta-\alpha)^{-1} \big\|_{L^2\to L^2}\leq =\sup_{\sigma > 0}\Big(\frac{(\sigma + \beta)^2}{(\sigma-\beta)^2+\gamma^2}\Big)^{\demi}
\leq C\frac{\beta}{\gamma}\]
which is bounded by $C/\eta$ for some $C>0$. To prove \eqref{disttospecleq1}, we use the same argument as above to bound
\begin{multline}
\Big\| \big( 1- \phi(\Delta/|\alpha|) \big)(\Delta-\alpha)^{-1} \Big\|_{L^p\to L^{p'}}\leq C\Big\| \big(1-\phi(\Delta/|\alpha|)\big) (\Delta + \beta)(\Delta-\alpha)^{-1} \Big\|_{L^2\to L^2} \\ \leq C  \sup_{\sigma > 0} \big( 1- \phi \big(\frac{\sigma}{\sqrt{\beta^2 + \gamma^2}}\big) \big)\Big(\frac{(\sigma + \beta)^2}{(\sigma-\beta)^2+\gamma^2}\Big)^{\demi}.
\ilabel{12}\end{multline}
Since $\eta>0$ can be taken as small as we like in the proof above,
it remains to consider the case where  
$\alpha=\beta(1\pm i\eps)$ for all sufficiently small $\eps  \geq 0$. We then take $|\eps|$ sufficiently small relative to $\delta$, and  then using the property that $\phi = 1$ on $[1 - \delta/2, 1 + \delta/2]$, we see that 
 \[ 
 \sup_{\sigma > 0} \big( 1- \phi \big(\frac{\sigma}{\sqrt{\beta^2 + \gamma^2}}\big) \big)\Big(\frac{(\sigma + \beta)^2}{(\sigma-\beta)^2+\gamma^2}\Big)^{\demi}
\leq 
 \sup_{|\sigma/\beta-1|>\delta/4}\Big(\frac{(\sigma/\beta +1)^2}{(\sigma/\beta-1)^2+\eps^2 }\Big)^{\demi} \leq C'.\]
This achieves the proof.
 \end{proof}

\begin{remark}
Using interpolation between \eqref{delta-alpha} (or \eqref{disttospecleq1}) and 
the $L^2\to L^2$ norm of $(\Delta-\alpha)^{-1}$, we have for each $\eta>0$ that there is $C>0$ 
such that for all $q\in [\frac{2n}{n+2},2]$ with $\frac{1}{q}+\frac{1}{q'}=1$ and 
\begin{equation}\ilabel{interp}
\begin{gathered}
\|(\Delta-\alpha)^{-1}\|_{L^q\to L^{q'}}\leq C|\alpha|^{n(\frac{1}{q}-\demi)-1}, \,\,\, \forall \, \alpha\in \cc, |\arg(\alpha)|>\eta \\
\| \big( 1- \phi(\Delta/|\alpha|) \big)(\Delta-\alpha)^{-1} \|_{L^q\to L^{q'}}\leq C|\alpha|^{n(\frac{1}{q}-\demi)-1}, 
\,\,\, \forall \, \alpha\in \cc,\, {\rm Re}(\alpha)>0 .
\end{gathered}
\end{equation}
\end{remark}

We next prove that if we look at the \emph{difference} between the resolvents  $(\Delta - \alpha)^{-1}$ and $(\Delta - \overline\alpha)^{-1}$, this is uniformly bounded $L^p \to L^{p'}$ as we approach the real axis. 
\begin{coro}\ilabel{cor:diff} There exists $C>0$ such that for all $\alpha\in\cc$ with $\Re \alpha > 0$ and all 
$q\in[\frac{2n}{n+2},\frac{2(n+1)}{n+3}]$:
\begin{equation} \ilabel{resdiffest}
||(\Delta-\alpha)^{-1} - (\Delta - \overline\alpha)^{-1} ||_{L^q\to L^{q'}}\leq C|\alpha|^{n(\frac{1}{q}-\frac{1}{2})-1}. 
\end{equation}
\end{coro}
\begin{proof} In view of \eqref{disttospecleq1}, it only remains to check that this is true when $\Delta/|\alpha|$ is spectrally localized near $1$. Writing $\alpha = \beta + i\gamma$ as above, we can estimate using \eqref{spe}
\begin{equation}\begin{gathered}
\Big\| \phi(\Delta/|\alpha|) \Big( (\Delta - \alpha)^{-1} - (\Delta - \overline\alpha)^{-1} \Big) \Big\|_{L^q \to L^{q'}} \\
\leq C \int_0^\infty \phi\big( \frac{\lambda^2}{|\alpha|}\big) \,  \Big| \frac1{\lambda - \beta - i\gamma} - \frac1{\lambda - \beta + i\gamma} \Big| \,  \big\| dE_{\sqrt{\Delta}}(\lambda) \big\|_{L^q \to L^{q'}} \, d\lambda \\
\leq C \int_0^\infty \phi\big( \frac{\lambda^2}{|\alpha|}\big) \frac{2|\gamma|}{(\lambda^2 - \beta)^2 + \gamma^2} \, \lambda^{n(\frac{1}{q}-\frac{1}{q'})} \, \frac{d\lambda}{\la}
\end{gathered}\end{equation}
and it is easy to check that $|\alpha|^{1-\ndemi(\frac{1}{q}-\frac{1}{q'})}$ times 
this integral depends only on $\gamma/\beta$ and is uniformly bounded as $\gamma/\beta \to 0$. We can then use \eqref{interp} to conclude.
\end{proof}

In particular if $q=2n/(n+2)$ is the Sobolev exponent, then $n(\frac{1}{q}-\frac{1}{2})-1=0$ and \eqref{resdiffest} is uniformly bounded 
in $\alpha$.

%%%%%%%%%%%%%%%%%%%%%%%%%%%%%%%%%
%%%%%%%%%%%%%%%%%%%%%%%%%%%%%%%%%
%%%%%%%%%%%%%%%%%%%%%%%%%%%%%%%%%
%%%%%%%%%%%%%%%%%%%%%%%%%%%%%%%%%

\section{Localized uniform estimates near the spectrum}\ilabel{sec:prelim}
Our aim to complete the proof of Theorem \ref{mainth}Ê is to
obtain a uniform estimate 
\begin{equation}\label{toprove}
||\phi(\Delta/|\alpha|)(\Delta-\alpha)^{-1}||_{L^p\to L^{p'}}\leq C\end{equation} 
when $|\arg \alpha|$ is small and $\phi\in C_0^\infty(\rr)$ is a function as in Proposition \ref{outsidecone}. 
We are not able to prove this directly, but rather we will later introduce a operator partition of unity, 
${\rm Id}=\sum_{i} Q_i(\alpha)$, depending on $|\alpha|$ and prove localized estimates  
\[|| Q_i(\alpha) \phi(\Delta/|\alpha|)(\Delta-\alpha)^{-1} Q_j(\alpha)||_{L^p\to L^{p'}}\leq C.\]

For that purpose, we start with an abstract result: 
Assuming that $Q(\alpha),Q'(\alpha)$ are bounded operators on $L^2$ such that 
pointwise bounds of the type \eqref{kernelest} Êare valid for $Q(\alpha)dE_{\sqrt{\Delta}}(\la)Q'(\alpha)$,
we show in the following Lemma a uniform estimate for the localized resolvent
\[|| Q(\alpha) \phi(\Delta/|\alpha|)(\Delta-\alpha)^{-1} Q'(\alpha)||_{L^p\to L^{p'}}\leq C.\]
This will be used with $Q(\alpha), Q'(\alpha)$ elements of a well chosen operator partition of 
 unity constructed in Section~\ref{sec:mpou}.

\begin{lemma}\ilabel{uniformonspect}
Let $\delta,\eta>0$ be small, let $\alpha\in \cc^*$ and $\phi\in C_0^\infty(\rr)$ as in Proposition \ref{outsidecone}. 
Assume that there exist $L^2$ bounded operators $Q(\alpha),Q'(\alpha)$ and $C>0$ such that 
\begin{equation}\ilabel{assumption}
 \Big| Q(\alpha)\pl^j_{\la}dE_{\sqrt{\Delta}}(\la)Q'(\alpha) (z,z') \Big|\leq C\la^{n-1-j}(1+\la d(z,z'))^{-\frac{n-1}{2}+j}
\end{equation}
for all $j\leq n/2$ and all $\la\in[(1-\delta)\sqrt{|\alpha|},(1+\delta)\sqrt{|\alpha|}]$. Then  there is $C'>0$ such that for all $\alpha\in \cc$
with $0<|\arg(\alpha)|\leq \eta$
\[ \Big\| Q(\alpha)\phi(\Delta/|\alpha|)(\Delta-\alpha)^{-1} Q'(\alpha)  \Big\|_{L^p\to L^{p'}}\leq C'.\] 
\end{lemma}
\begin{proof} Let us prove the result for the case $\arg(\alpha)<0$; the other case is similar. 
We shall prove this by complex interpolation. Let us consider the analytic family of operators 
$H_{s,\alpha}(\sqrt{\Delta/|\alpha|})$ in $\Re(s)\leq 0$ if  
\[H_{s,\alpha}(x):= e^{s^2}|\alpha|^{s}\phi(x^2)(x^2-e^{i\arg(\alpha)})^s\]
and the logarithm is defined with a cut at $\rr^-$. Since the operator we are interested in is
$H_{-1,\alpha}(\sqrt{\Delta/|\alpha|})$ and since by the spectral theorem
\[\sup_{t\in \rr}||H_{it,\alpha}(\sqrt{\Delta/|\alpha|})||_{L^2\to L^2}\leq C,\]  
it suffices to prove that 
\begin{equation}\ilabel{toshow}
\sup_{t\in \rr}||H_{-\ndemi +it,\alpha}(\sqrt{\Delta/|\alpha|})||_{L^1\to L^\infty}\leq C
\end{equation}
and the result follows by complex interpolation for the family of operators
$H_{s,\alpha}(\sqrt{\Delta/|\alpha|})$.

Let us first assume that $n$ is odd. Then we write 
\[\begin{split}
\frac{H_{-\ndemi+it,\alpha}(\sqrt{\Delta/|\alpha|})}{e^{(it-\ndemi)^2}}=&|\alpha|^{-\ndemi+it}
\int_0^\infty\phi(\frac{\la^2}{|\alpha|})(\frac{\la^2}{|\alpha|}-e^{i\arg(\alpha)})^{-\ndemi+it}dE_{\sqrt{\Delta}}(\la)d\la\\
=&\frac{|\alpha|^{-\frac{n-1}{2}+it}}{L(t)}
\int_0^\infty\phi(\la^2)(\frac{1}{2\la}\pl_{\la})^{\frac{n-1}{2}}
\Big(\la^2-e^{i\arg(\alpha)}\Big)^{-\demi+it}dE_{\sqrt{\Delta}}(\la|\alpha|^\demi)d\la
\end{split}
\]
where $L(t)$ is a polynomial such that $|L(t)|>C>0$ for all $t\in\rr$. We compose with $Q(\alpha)$ on the left and $Q'(\alpha)$ on the right and integrate by parts by using the vanishing order at $\la=0$ in 
\eqref{assumption}, this yields 
\[\begin{gathered}
e^{-(it-\ndemi)^2}Q(\alpha)H_{-\ndemi+it,\alpha}(\sqrt{\Delta/|\alpha|})Q'(\alpha)=\\
\frac{|\alpha|^{-\frac{n-1}{2}+it}}{L(t)}
\int_0^\infty
(\la^2-e^{i\arg(\alpha)})^{-\demi+it}Q(\alpha)(\pl_{\la}\frac{1}{2\la})^{\frac{n-1}{2}}
\Big(\phi(\la^2)dE_{\sqrt{\Delta}}(\la|\alpha|^\demi)\Big)Q'(\alpha)d\la
\end{gathered}
\]
Using the estimate \eqref{assumption} with $j\leq \frac{n-1}{2}$, we deduce the bound
\[ \Big\|Q(\alpha)(\pl_{\la}\frac{1}{2\la})^{\frac{n-1}{2}}
\Big(\phi(\la^2)dE_{\sqrt{\Delta}}(\la|\alpha|^\demi)\Big)Q'(\alpha) \Big\|_{L^1\to L^\infty}\leq C|\alpha|^{\frac{n-1}{2}}\]
and therefore 
\[ \Big\| Q(\alpha)H_{-\ndemi+it,\alpha}(\sqrt{\Delta/|\alpha|})Q'(\alpha) \Big\|_{L^1\to L^\infty}\leq 
Ce^{\pi|t|-t^2}\int_{0}^2|\la^2-1|^{-\demi}d\la\leq C.\]

We now want to deal with the case $n$ even. Let us write, using integration by parts ($n/2$ times) as before,
\[\frac{H_{-\frac{n+1}{2}+it,\alpha}(\sqrt{\Delta/|\alpha|})}{e^{(it-\frac{n+1}{2})^2}}=\frac{|\alpha|^{-\frac{n}{2}+it}}{L(t)}
\int_0^\infty(\la^2-e^{i\arg(\alpha)})^{-\demi+it}(\pl_{\la}\frac{1}{2\la})^{\frac{n}{2}}
\Big(\phi(\la^2)dE_{\sqrt{\Delta}}(\la|\alpha|^\demi)\Big)d\la
\] 
for some polynomial $L(t)$ such that $|L(t)|>C>0$ for all $t\in\rr$. 
We multiply this by $Q(\alpha)$ on the left and $Q'(\alpha)$ on the right and use \eqref{assumption} to obtain the pointwise estimate
\begin{equation}\ilabel{estimateleft}
|Q(\alpha)H_{-\frac{n+1}{2}+it,\alpha}(\sqrt{\Delta/|\alpha|})Q'(\alpha)(z,z')|\leq C|\alpha|^{-\demi}(1+(1+\delta)|\alpha|^\demi d(z,z'))^\demi.
\end{equation}
Similarly, we have by integration by parts ($n/2-1$ times)
\[\frac{H_{-\frac{n-1}{2}+it,\alpha}(\sqrt{\Delta/\alpha})}{e^{(it-\frac{n-1}{2})^2}}=\frac{|\alpha|^{-\frac{n}{2}-1+it}}{L(t)}
\int_0^\infty(\la^2-e^{i\arg(\alpha)})^{-\demi+it}(\pl_{\la}\frac{1}{2\la})^{\frac{n}{2}-1}
\Big(\phi(\la^2)dE_{\sqrt{\Delta}}(\la|\alpha|^\demi)\Big)d\la
\] 
for some polynomial $L(t)$ as before, and using  \eqref{assumption} we get 
\begin{equation}\ilabel{estimateright}
|Q(\alpha)H_{-\frac{n-1}{2}+it,\alpha}(\sqrt{\Delta/|\alpha|})Q'(\alpha)(z,z')|\leq C|\alpha|^{\demi}\Big(1+(1-\delta)|\alpha|^\demi d(z,z')\Big)^{-\demi}.
\end{equation}
Since for each $(z,z')$ the Schwartz kernel 
$Q(\alpha)H_{s,\alpha}(\sqrt{\Delta/|\alpha|})Q'(\alpha)(z,z')$ is holomorphic in $s$ in the strip $|{\rm Re}(s)+n/2|\leq 1/2$, we can use Phragmen-Lindelof and \eqref{estimateleft}, \eqref{estimateright}Ê to deduce that 
\[ \Big|Q(\alpha)H_{-\frac{n}{2},\alpha}(\sqrt{\Delta/|\alpha|})Q'(\alpha)(z,z') \Big|\leq C
\sup_{w>0}\Big(\frac{1+(1+\delta)w}{1+(1-\delta)w}\Big)^{\frac{1}{4}}\leq C. \]
This ends the proof.
\end{proof}

\begin{remark}\ilabel{2(n+1)/(n+3)} The same type of proof shows that in fact, or all $q\in [\frac{2n}{n+2},\frac{2(n+1)}{n+3}]$ 
and $\frac{1}{q'}+\frac{1}{q}=1$
there is $C>0$ such that for all $\alpha>0$ 
\begin{equation}\ilabel{endpoint}
\Big\| Q(\alpha)\phi(\Delta/\alpha)(\Delta-\alpha\pm i0)^{-1} Q'(\alpha)  \Big\|_{L^q\to L^{q'}}\leq C\alpha^{n(\frac{1}{q}-\demi)-1}.
\end{equation}
under assumption \eqref{assumption} for $j\leq (n+1)/2$. Let us give a brief argument, which is due to Adam Sikora whom we gratefully acknowledge. We want to interpolate the norms of the operator
 $H_{s,\alpha}(\sqrt{\Delta/\alpha})$ between 
${\rm Re}(s)=-\frac{n+1}{2}$ (for $L^1\to L^\infty$) and ${\rm Re}(s)=0$ (for $L^2\to L^2$).
First, let $\chi_+^z$ be the family of distributions on $\rr$ defined by
analytic continuation of $\chi_+^z(\la):=\la_+^{z}/\Gamma(z+1)$ in $z\in \cc$; in particular $\chi^{-k}_+=\delta^{k-1}_{0}$ if $k\in\nn$. Then
the following estimate holds for all $a<b<c<0$ and $b=\theta a+(1-\theta)c$: there is $C>0$ such that for all $f\in C_0^\infty(\rr)$ and $t\in\rr$, 
\begin{equation}\ilabel{Adamtrick} 
||(\la\pm i0)^{b+it}* f||_{L^\infty}\leq C(1+|t|)e^{\frac{\pi}{2}|t|}||\chi_+^{a}* f||_{L^\infty}^{\theta}||\chi_+^{c}* f||^{1-\theta}_{L^\infty}.
\end{equation}
The proof is an exercise, and is very similar to the proof of Lemma 3.3 in \cite{GHS2}.
To get estimate on the $L^\infty$ norm of the Schwartz kernel of $H_{s,\alpha}(\sqrt{\Delta/\alpha})$ on the line ${\rm Re}(s)=-\frac{n+1}{2}$, we write as kernels
\[H_{s,\alpha}(\sqrt{\Delta/\alpha})(z,z')=e^{s^2}\alpha^{s+\demi}\Big\cjg (\la-1+i0)^{s}, \frac{\phi(\la)}{2\sqrt{\la}}Q(\alpha)dE_{\sqrt{\Delta}}(\alpha\sqrt{\la})Q'(\alpha)(z,z')\Big\cjd\] 
where the pairing is the distribution pairing in $\la$. If $n$ is odd, we set $a=-\frac{n+3}{2}$, $b=-\frac{n+1}{2}$ and $c=-\frac{n-1}{2}$
with $a,c$ negative integers and apply \eqref{Adamtrick} together with the estimates \eqref{assumption} (with $j=-a-1$ and 
$j=-c-1=-a-3$), to deduce that there is $C>0$ such that for all $t\in\rr$
\[|H_{b+it,\alpha}(\sqrt{\Delta/\alpha})(z,z')|\leq C\alpha^{-\demi}\sup_{w>0}\Big(\frac{1+(1+\delta)w}{1+(1-\delta)w}\Big)^\demi\leq C \alpha^{-\demi}.\]
Using the complex interpolation Lemma in \cite[p 385]{St}, this gives \eqref{endpoint} for $q=\frac{2(n+1)}{n+3}$ and the other $q$ are obtained by interpolating with $q=p$. In even dimension, the similar argument works with $a=-\frac{n+2}{2}$, $b=-\frac{n+1}{2}$ and $c=-\frac{n}{2}$. 
\end{remark}

\section{Microlocal partition of unity}\ilabel{sec:mpou}

\subsection{Pseudodifferential calculus}\ilabel{psido} 
We start with some preliminaries on pseudo-differential operators in our setting. The natural class of pseudo-differential operators for this geometry is the scattering calculus introduced by Melrose \cite{Mel}.
The semi-classical version (high-frequency) is defined in Vasy-Zworski \cite{VaZw} and in 
details in Appendix A in Wunsch-Zworski \cite{WuZw}. 
We will give a brief review of this calculus and refer the reader to these articles, for further details 
(see also\cite[Sec. 2]{Da} for wave-front sets discussions). The Euclidean version of the semiclassical 
calculus can be found for instance in the book of Zworski \cite{Zw}.\\

\noindent\textbf{Phase space.}
Recall that $g$ is an asymptotically conic metric on $M$. In the case that it is exactly conic near infinity, it takes the form $dx^2/x^4 + h/x^2$ when $x$ is small, or equivalently $dr^2 + r^2 h$ when $r = 1/x$ is large, with $h$ independent of $x$. A frame of uniformly bounded vector fields (with respect to $g$) is given by $x^2 \partial_x$ and $x \partial_{y_i}$ for small $x$; dually, a uniformly bounded coframe is given by $dx/x^2$ and $dy_i/x$. These scalings motivate the introduction of the \emph{scattering cotangent bundle} as the natural phase space for dynamics (such as geodesic flow) on $(M, g)$. 
The scattering cotangent bundle 
${^\SC T^*}\bbar{M}$ is a smooth bundle over $\bbar{M}$ defined as follows : let $\eps>0$ be small, then  
over $\{x\geq \eps\}$, ${^\SC T^*}\bbar{M}$ is simply $T^*M|_{x\geq \eps}$, 
and over $\{x\leq 2\eps\}$ its smooth sections are given by linear combinations over  $C^\infty(M)$ 
of  $dr = -dx/x^2$ and $\omega/x$ where $\omega$ are $1$-forms smooth up to $\pl \bbar{M}$. In local coordinates $(x,y_1,\dots,y_{n-1})$ near the boundary 
(where $y_i$ are local coordinates on $\pl\bbar{M}$),  the bundle ${^\SC T^*}\bbar{M}$ is locally spanned by
\[\frac{dx}{x^2}, \, \frac{dy_1}{x}, \, \dots,\, \frac{dy_{n-1}}{x}.\]
Locally near a point of $\pl\bbar{M}$, 
we use the coordinates  for a point $\xi\in {^\SC T}^*\bbar{M}$ 
\[\xi= \nu d\big(\frac{1}{x} \big)+ \mu.\frac{dy}{x}.\]
Thus $(\nu, \mu)$ form linear coordinates on each fibre of ${^\SC T^*}\bbar{M}$ near the boundary. 
For example, if $M = \RR^n$, with Euclidean coordinate $z$, then $(\nu, \mu)$ are the radial and angular components of the cotangent variable $\zeta$ dual to $z$. 
The geodesic flow for the metric $g$ acts on ${^\SC T^*}\bbar{M}$ and preserve the energy levels $|\xi|_g={\rm const}$. We 
shall use the notation $g^t$ for the geodesic flow at time $t$.\\

\noindent\textbf{Symbols.} 
Let $h_0>0$ and $h\in(0,h_0]$ the semi-classical parameter. A (semi-classical) 
symbol $a$ on $\rr^n$ in the class $S^{m,\ell,k}(\rr^{n})$ with $(m,l,k)\in\rr^3$ is defined to be a smooth function on $[0,h_0)_h\x \rr^{2n}$
satisfying: for all $\alpha,\beta$ multi-indices,  there exists $C_{\alpha,\beta}$ such that for all $h,z,\zeta$ 
\[|\pl_z^\alpha\pl^\beta_\zeta a(h,z,\zeta)|\leq C_{\alpha,\beta} h^{-k}(1+|z|)^{-\ell-|\alpha|}(1+|\zeta|)^{m-|\beta|}.\] 
A symbol is \emph{classical} if 
$h^{k}\kappa^{m}x^{-\ell}a\in C^\infty([0,h_0)\x S^n_+\x S^n_+)$ where $S^n_+$ is the radial smooth compactification of $\rr^n$ and $x=1/|z|$, $\kappa=1/|\zeta|$. 
Similarly one can fiber radially compactify the cotangent space ${^\sca T}^*\bbar{M}$ into ${^\sca \bbar{T}}^*\bbar{M}$ 
and a classical 
symbol in $S^{m,\ell,k}(M)$ is defined to be a function $a$ on $[0,h_0)\x {^\sca \bbar{T}}^*\bbar{M})$ such that   $h^{k}\kappa^{m}x^{-\ell}a\in C^\infty([0,h_0)\x {^\sca \bbar{T}}^*\bbar{M})$
where $\kappa$ is a boundary defining function of the fiber infinity.  

One can also define $S^{m,\ell,k}(M)$  by reducing to the $\rr^n$ case.
Consider an neighbourhood of a point $y_0 \in \partial\bbar{M}$ of the form $\{ (x, y) \in M \mid x < \epsilon, y \in U \}$ where $U$ is a neighbourhood of $y_0$ in $\partial \bbar{M}$. Choosing a diffeomorphism $\omega$ from $U$ to an open subset $U'$ in $S^{n-1}$, we map $(x, y)$ to $\omega(y)/x \in \RR^n$ using the standard embedding of $S^{n-1}$ in $\RR^n$. We call this a `conic type' chart. Such charts induce a smooth map between the associated scattering cotangent bundles, allowing one to define classical symbols in $S^{m,\ell,k}(M)$ as those that can be pulled back from $S^{m,\ell,k}(\rr^{n})$ via such conic-type charts. 

We shall only consider classical symbols in what follows and, by abuse of notation, 
we denote by $S^{m,\ell,k}(M)$ the class of classical  symbols of order $(m,\ell,k)$ on $M$.
There is a principal symbol map $\sigma: S^{m,\ell,k}(M)\to S^{m,\ell,k}(M)/S^{m-1,\ell+1,k-1}(M)$ which assigns the leading 
term in the asymptotic expansion as $h\kappa x \to 0$, ie. at the boundary of $[0,1)\x {^\sca \bbar{T}}^*\bbar{M}$.\\

\noindent\textbf{Quantization.} Let $\dot{C}^\infty(\bbar{M})$ be the space of smooth functions on $\bbar{M}$ vanishing
to infinite order at $\pl\bbar{M}=\{x=0\}$, and $C^{-\infty}(\bbar{M})$ its dual.
We say that an operator $A:\dot{C}^\infty(\bbar{M})\to C^{-\infty}(\bbar{M})$ is in $\Psi^{m,\ell,k}(M)$ if it can be written, 
up to a residual operator (ie. mapping $C^{-\infty}(\bbar{M})$ to $h^\infty \dot{C}^\infty(\bbar{M})$), 
as a finite sum of operators $A_j$ with Schwartz kernels supported in $U_j\x U_j$ in the Euclidean charts $U_j$ with coordinates $z$ 
(of conic type near infinity and relatively compact otherwise) and of the form 
\[A_ju(z)=\frac{1}{(2\pi h)^n}\int_{\rr^n}e^{i\frac{(z-z').\zeta}{h}}a_j(h,z,\zeta)u(z)dzd\zeta, \quad z\in U
\]
for $u$ supported in the chart and $a_j\in S^{m,\ell,k}(\rr^n)$. 
There is a well defined principal symbol map $\sigma:\Psi^{m,\ell,k}(M)\to S^{m,\ell,k}(M)/S^{m-1,\ell+1,k-1}(M)$. 
We can also define a (semiclassical) quantization ${\rm Op}_h$  such that ${\rm Op}_h:S^{m,\ell,k}(M)\mapsto \Psi^{m,\ell,k}(M)$ for all $(m,\ell,k)$; this can be done by choosing a set of  coordinates, two associated partitions of unity $\sum_j \varphi_j=1=\sum_j\psi_j$ with $\psi_j\varphi_j=\varphi_j$, and then 
set ${\rm Op}_h(a)=\sum_j \psi_j A_j\varphi_j$ where we use the formula above for $A_j$ with $a_j$ the pullback of $a$ in the 
chart. The space $\Psi^{*,*,*}(M)$ forms an algebra and 
$\Psi^{m,\ell,k}(M).\Psi^{m',\ell',k'}(M)\subset \Psi^{m+m',\ell+\ell',k+k'}(M)$.
The principal symbol is multiplicative: 
$\sigma(AB)=\sigma(A)\sigma(B)$. An operator $A\in \Psi^{m,\ell,0}(M)$ is said to be elliptic at $K\subset {^\sca \bar{T}^*M}$ if 
$\sigma^{m}x^{\ell}\sigma(A)\geq c>0$ on $K$ for some constant $c$.
One also has  that if $A\in\Psi^{m,\ell,k}(M), B\in \Psi^{m',\ell',k'}(M)$, then $[A,B]\in \Psi^{m+m-1,\ell+\ell'+1,k+k'-1}(M)$ 
and 
\begin{equation}\ilabel{commut}
\sigma([A,B])=\frac{h}{i}\{\sigma(A),\sigma(B)\}
\end{equation}  
where $\{a,b\}$ is the Poisson bracket.\\ 

\noindent\textbf{Wave-front sets.}
The wave front set $\WF'(A)$ of $A={\rm Op}(a)$
is a subset of the boundary of $[0, h_0) \times {^\sca \bbar{T}}^*\bbar{M}$, defined as the 
 complement in $\partial \big([0, h_0) \times {^\sca \bbar{T}}^*\bbar{M} \big)$ of those points $\xi$ 
 such that in a neighborhood of $\xi$, $a$ vanishes to all orders  at all boundary 
faces of $[0,h_0)\x {^\sca \bbar{T}}^*\bbar{M}$. In other words, $\pl^{\alpha}a =\mc{O}(h^Nx^N\kappa^{N})$ for 
all $N\in \nn$ and all multi-indices $\alpha$ near $\xi$ (here $\kappa$ is a defining function for 
the fiber boundary, in Euclidean local coordinates $(z,\zeta)$ we can take $\kappa=1/|\zeta|$). This notion of $\WF'$ is globally 
well defined and 
\begin{equation}\ilabel{WFsets}
\WF'(AB)\subset \WF'(A)\cap \WF'(B).\end{equation}
A function $u\in x^{-N}H^{-N}(M)$ with  $||u||_{H^{-N}}=\mc{O}(h^{-N})$ for some $N\in\nn$ 
is said to be \emph{tempered}. We define its wave-front set  
to be the complement of the set of points $\xi\in \partial \big([0, h_0) \times {^\sca \bbar{T}}^*\bbar{M} \big)$ such that there exists $A\in \Psi^{m,\ell,0}(M)$ elliptic
near $\xi$ such that $||x^{-N}Au||_{L^2}=\mc{O}(h^\infty)$ for all $N\in\nn$. In particular, we have 
\[\WF'(Au)\subset \WF'(A)\cap \WF'(u)\]
if $A\in \Psi^{m,\ell,k}(M)$.\\

\noindent\textbf{Geodesic flow near infinity}.
We let $p$ be the Hamiltonian on ${^\SC T}^*\bbar{M}$ defined by $p(m,\xi)=g_m(\xi,\xi)$.
If the metric $g$ is written near $\pl\bbar{M}$ under the form $g=dx^2/x^4+h_x/x^2$ where $h_x$ is a smooth family of metric on $\pl\bbar{M}$
for $x\in[0,\eps)$, then the principal symbol of $\Delta_g$ near $\pl\bbar{M}$ is 
\[p(x,y;\xi)=|\xi|^2_g=\nu^2+ |\mu|^2_{h_x}\]
and the Hamiltonian vector field associated to $p$ is (see \cite{VaZw}) 
\[H_p= -2x\nu (x\pl_x) +(2 x|\mu|^2_{h_x}-x^2\pl_x|\mu|^2_{h_x})\pl_\nu-2x\nu \mu.\pl_\mu+xH_{|\mu|^2_{h_x}}\]
where $H_{|\mu|^2_{h_x}}$ is the Hamiltonian vector field of $|\mu|^2_{h_x}$ on $\pl\bbar{M}$.
Near $x=0$, the equation for the geodesic flow is given by (dot denotes time derivative)
\begin{equation}\ilabel{floweq}
\begin{gathered}
\dot{x}(t)= -2x(t)^2\nu(t), \quad \dot{\nu}(t)=2x(t)|\mu(t)|^2_{h_x}-x(t)^2\pl_x(|\mu|^2_{h_x})(t)\\
\dot{\mu}(t)=-2x(t)\nu(t)\mu(t)-x(t)\pl_y(|\mu|^2_{h_x})(t), \quad \dot{y}(t)=2x(t)(\pl_{\mu}|\mu|^2_{h_x})(t).
\end{gathered}\end{equation}
We denote the Hamiltonian flow at time $t$ by $g^t$.
As described in \cite{Mel,MZ}, the Hamiltonian $H_p$ can be written near $\pl\bbar{M}$ as 
$H_p=x\hat{H}_p+x^2V$ with $V$ smooth on ${^\SC T}^*\bbar{M}$ and tangent to  its boundary, and  
\begin{equation}\ilabel{hatHp}
\hat{H}_p:= -2\nu (x\pl_x+\mu.\pl_\mu) +2 |\mu|^2_{h_0}\pl_\nu+H_{|\mu|^2_{h_0}}.
\end{equation}
This last vector field is tangent to $\{x=0,p=1\}\cap{^\SC T}^*\bbar{M}$ and vanishes there at $\mu=0$ only. 
Moreover one has $\dot{\nu}(t)>0$ along its integral curve as long as $|\nu(t)|<1$.\\

\noindent\textbf{Outgoing/incoming relations}. Let $Q,Q'\in \Psi^{-\infty,0,0}(M)$ be two pseudodifferential operators. 
We say that $Q$ is  \emph{not  outgoing-related to $Q'$} if the forward flowout from $\WF'(Q')$ 
by the geodesic flow does not meet $\WF'(Q)$, 
that is if $g^t(\WF'(Q'))\cap \WF'(Q)=\emptyset$ for all $t\geq 0$. 
For boundary points $\xi\in {^\SC T}_{\pl \bbar{M}}^*\bbar{M}$, the action of the flow $g^t$ on $\xi$ 
is defined to be the flow at time $t$ for the vector field $\hat{H}_p$ acting on ${^\SC T}_{\pl \bbar{M}}^*\bbar{M}$. 
Similarly we say that $Q$ is  \emph{not  incoming-related to $Q'$} if the backward flowout from $\WF'(Q')$ 
by the geodesic flow does not meet $\WF'(Q)$, or equivalently if $Q'$ is not outgoing-related to $Q$.

\subsection{The microlocal partition of unity}\ilabel{micpart}

Let $\alpha=\beta+i\gamma$ with $\beta>\beta_0$  for some fixed $\beta_0$ and let $h=1/\sqrt{|\alpha|}$ be our semi-classical parameter. 
We will construct an operator partition of unity ($J\in\nn$ is independent of $h$)
\[{\rm Id}=(1-\phi(h^2\Delta))+
\sum_{i=1}^J Q_j(h) , \quad \textrm{ with }Q_j(h)\in \Psi^{-\infty,0,0}(M)\] 
for which there is a dichotomy: either $Q_i(h)$ and $Q_j(h)$ have close support and
such that \eqref{kernelest} holds, or they are such that $Q_i(h)$ is either not outgoing-related or not incoming-related to $Q_j(h)$. 
In Proposition~\ref{microlocaltrivial}, we will show that $Q_i(h) (h^2\Delta-(\beta \pm i\gamma))^{-1} Q_j(h)$ ($+$ for not outgoing-related, $-$ for not incoming-related) is a trivial operator for 
 $|\gamma|\leq \eta$ and $|\beta-1|\leq \delta$.    \\
 
 \begin{comment}
 we will prove that its Schwartz 
kernel is smooth  with a pointwise bound of the form $C_N(xx'h)^N$ for all $N>0$, $\beta\in(1-\delta,1+\delta)$ 
and $|\gamma|\leq \eta$, 
from which we trivially conclude the uniform estimate
\begin{equation}
 \big\| Q_i(h)(h^2\Delta- \beta- i\gamma)^{-1}Q_j(h) \big\|_{L^p\to L^{p'}}\leq C 
\ilabel{dichotomy} \end{equation}
for all $\beta\in(1-\delta,1+\delta)$ and $|\gamma|\leq \eta$.\\

We will describe the geodesic flow near infinity and construct a certain partition of unity of the energy shell 
$\{||\xi|_g-1|\leq \delta\}$. We use a product decomposition $[0,\eps)_x\x \pl\bbar{M}$ for a collar near $\pl\bbar{M}$ in which
the metric has the form \eqref{metric}.  We shall denote $\pl_x$ 
the vector field near $\pl\bbar{M}$ induced by the product decomposition. \\
\end{comment}

\noindent\textbf{Partition in phase space.} We first need the following  
\begin{lemma}\ilabel{partitioninx>eps}
Let $N\in \nn$ be large and $\eps>0,\delta>0$ be small with $\delta<1/N$. 
Let us cover $[-(1+\delta),1+\delta]$ by $N$ closed intervals $B_i$ 
of same size $3(1+\delta)/N$ and each $B_i$ intersects only $B_{i-1}$ and $B_{i+1}$, then 
we cover $p^{-1}([1-\delta,1+\delta])\cap \{x<\eps\}$ by $\cup_{i=1}^{N}\til{B}_i$ with 
$\til{B}_i:=p^{-1}(]1-\delta,1+\delta[)\cap \{x<\eps, \nu\in B_i\}$. 
If $N$ is large enough and $\delta,\eps>0$ are small enough, then \\
1) when ${\rm dist}(B_i,B_j)>0$ we have 
\[\textrm{ either }\forall t\geq 0,\,\, g^t(\til{B}_i)\cap \til{B}_j=\emptyset 
\textrm{ or }\forall t\leq 0,\,\, g^t(\til{B}_i)\cap \til{B}_j=\emptyset.\]
2) for any small open set $U$ in ${^\SC T}^*\bbar{M}\cap \{x\geq 2\eps\}$, then for all $i\leq N$
\[\textrm{ either }\forall t\geq 0,\,\, g^t(U)\cap \til{B}_i=\emptyset 
\textrm{ or }\forall t\leq 0,\,\, g^{t}(U)\cap \til{B}_i=\emptyset.\]
\end{lemma}
\begin{proof}

We start with a couple of facts about the flow near $\pl\bbar{M}$.
First, from \eqref{floweq}, there exists $\eps>0$ small so that as long as $x(t)\leq 2\eps$, one has  
\begin{equation}
\dot{\nu}(t)\geq (2-C\eps)x(t)|\mu|_{h_x}^2\geq 0.
\ilabel{nudot}\end{equation}
Also, the sign of $\dot x$ is the same as that of $-\nu$. Therefore, if $x(0) \leq 2\eps$ and $\nu(0) \geq 0$ then $x(t) \leq 2\eps$ for all $t \geq 0$, while if $x(0) \leq 2\eps$ and $\nu(0) \leq 0$ then $x(t) \leq 2\eps$ for all $t \leq 0$. Equivalently, the compact set $x\geq 2\eps$ is geodesically convex.

Second, suppose that $x(0) = \epsilon$ and $-\delta\leq \nu(0)\leq 0$ for some $0 < \delta < 1/8$. As above, $x(t) \leq \epsilon$ for all negative $t$. Now we show that $x(t)$ will be no larger than $3\eps/2$ for all positive $t$, provided that $\epsilon$ is sufficiently small (depending only on $(M,g)$). To see this, let $t_0>0$ be the time 
defined by $0=\nu(t_0)\geq \nu(t)$ for all $t\in(0,t_0)$; such a time exists by the non-trapping assumption which means that $\nu(t) \to 1$ as $t \to \infty$. For $t\in [0,t_0]$ one has
$|\mu(t)|^2 \geq1-\delta$.
Also, on the interval $[0, t_0]$, $x$ is nondecreasing, so  $x(t) \geq \eps$ for $t \in [0, t_0]$, while on the other hand 
$|\dot{x}/x^2| = |d/dt(x^{-1})| \leq 2$, implying that $x(t) \leq 2\eps$ 
for $t \leq 1/4\eps$.  Now using the equation for $\dot \nu$ we see that, for sufficiently small $\epsilon$, $\dot{\nu}(t)\geq 2\eps(1-\delta-C\eps) > \eps$ for $0<t<\min(1/4\eps,t_0)$. For  $\delta < 1/8$ this means that $\nu$ becomes nonnegative within time $\delta/\epsilon < 1/8\eps$, i.e. $t_0 < \delta/\eps$. This implies that $x(t_0) < 3\eps/2$ for $\delta < 1/8$. Now we have $x(t_0) < 3\epsilon/2$ and $\nu(t_0) = 0$ we see from the paragraph above that $x(t) < 3\eps/2$ for all times $t \in \RR$. A similar argument shows that if $x(0) = \epsilon$ and $0 \leq \nu(0) \leq \delta$ for some $0 < \delta < 1/8$, then $x(t) < 3\eps/2$ for all times $t \in \RR$.\\

Let us now show 1). Without loss of generality assume that $\inf(B_i)>\sup(B_j)$. Now consider a geodesic $\gamma$ starting at  $\gamma(0) \in \til{B}_i$. If $\gamma(t)$ stays entirely within $x \leq 3\epsilon/2$ for $t \geq 0$, then $\nu(t)$ is nondecreasing along $\gamma$ for $ t\geq 0$ and it follows that $\gamma(t)$ is  disjoint from $\til{B}_j$ for $t \geq 0$. On the other hand, if $\gamma(t)$ reaches $x > 3\epsilon/2$ for some $t_2 > 0$, then it follows that $\nu(0) < 0$, so $\inf(B_i) < 0$. Let $(t_1, t_3)$ be the maximal open interval containing $t_2$ on which $x(t) > 3\eps/2$. Then we have $\nu(t)\leq 0$ for $t \leq t_1$ and $\nu(t) \geq 0$ for $t \geq t_3$ since $\nu$ is nondecreasing whenever $x \leq 2\eps$. It follows that $\gamma(t)$ is disjoint from $\til{B}_j$ for all $t \geq 0$: for $t \leq t_1$ since $\nu \geq \nu(0) > \sup {B}_j$, on $[t_1, t_3]$ since $x \geq 3\eps/2$ and on $t \geq t_3$ since $\nu \geq 0 \geq \nu(0) > \sup {B}_j$. \\

Showing 2) is  similar. Suppose that $\inf B_i > 0$. Then a trajectory $\gamma$ with $\gamma(0) \in \til{B}_i$ stays in $x \leq \epsilon$ for all $t \geq 0$ and hence is disjoint from $U$. Similarly if $\sup B_j < 0$, trajectories starting in $\til{B}_i$ stay in 
$x \leq \epsilon$ for all $t \leq 0$. Now consider the case that $0 \in B_i$. Then provided that $\delta < 1/8$ and $\epsilon$ is sufficiently small, the second fact above shows that $x(t) \leq 3\epsilon/2$ for all time, showing that trajectories starting in $\til{B}_i$ are disjoint from $U$ for all $t \in \RR$.
\end{proof}

In a second Lemma, we complete the adapted covering by covering the region $\{x\geq \eps\}$.
\begin{lemma}\ilabel{Dj}
Let $\iota>0$ be sufficiently small, in particular smaller than half of the injectivity radius of $(M,g)$, and $\delta>0$ be small. 
There exist  open sets $D_j\subset \{x> \eps/2\}\cap p^{-1}(]1-\delta,1+\delta[)$ for $j=1,\dots, N'$ 
such that $\{x\geq \eps\}\cap p^{-1}(]1-\delta,1+\delta[)\subset \cup_{j}D_j$ and satisfying the following properties:\\
1) If $D_j\cap D_k\not=\emptyset$,  then all geodesics going from a point $m\in D_j$ 
to a point $m'\in D_k$ have length less than $\iota$.\\
2) If $(j,k)$ are such that ${\rm dist}(D_j, D_k)>\iota$, then either $g^t(D_j)\cap D_k=\emptyset$ for all 
$t<\iota/2$ or  $g^t(D_j)\cap D_k =\emptyset$ for all $ t>-\iota/2$.\\
3) If $D_j\subset \{x< 2\eps\}$, then there exist $i\leq N$ such that 
$D_j\subset \hat{B}_i$ where $\hat{B}_i:=p^{-1}(]1-\delta,1+\delta[)\cap \{x<2\eps, \nu\in B_i\}$. 
\end{lemma}
\begin{proof}
We first cover the set $S^*M\cap\{x\geq \eps/2\}$ by balls $D_j$ of radius $r>0$ 
where $r$ is chosen smaller than $\iota/4$, and we thicken $D_j$ homogeneously in $|\xi|$ to make open sets in 
$\{||\xi|_g-1|<\delta ,x\geq \eps/2\}\subset T^*M$,  implying 1) for $\delta$ small. 
Such an $r$ exists by compactness. %(see Lemma 7.9 of \cite{GHS2}). 
By taking $r$ small enough depending only on $N$, it is also clear that 3) can be obtained. 
Using the non-trapping assumption, we will prove that 2) is satisfied if $r$ is chosen small enough. Taking $\delta>0$  small, 
it suffices to consider the flow on $S^*M$.
First, the region $\{x\leq \eps/2\}\cap S^*M$ is geodesically convex if $\eps>0$ is small enough, then since $g$ is non-trapping, 
we can define a function $z\to t(z)$ on $\{x> \eps/2\}\subset S^*M$ 
which to a point $z$ assigns the time $t(z)>0$ such that 
$x(g^{t(z)}(z))=\eps/2$, and this function is continuous ($t(z)$ is obtained by applying the implicit function theorem  to 
the function $(z,t)\to x(g^t(z))$), which implies that for any compact set $K\subset S^*M\cap \{x\leq \eps/2\}$, there exists  
$T_+>0$ such that $g^t(K)\in \{x<\eps/2\}\cap S^*M$  for all $t>T_+$. Similarly there is a $T_-<0$ so  that for all $t<T_-$, 
$g^t(K)\in \{x<\eps/2\}\cap S^*M$.
Now if such covering $D_j$ does not exist, 
we can construct, using compactness, two sequences of points $z_n,z_n'\in  \{x> \eps/2\}$ 
both converging to the same point $z\in  \{x> \eps/2\}\cap S^*M$, and two sequences of times 
$t_n\leq -\iota/2, t_n'\geq \iota/2$ such that $\lim_{n\to \infty} g^{t_n}(z_n)=\lim_{n\to \infty} g^{t'_n}(z'_n)$. 
Moreover, these times are necessarily bounded by the existence of $T_\pm$ if $K$ is chosen to be a small ball centered at $z$.
Passing to a subsequence, $t_n$ and $t'_n$ have an accumulation point $t\leq -\iota/2$, resp. $t'\geq \iota/2$, 
then we have $g^t(z)=g^{t'}(z)$, which implies that $z$ is on a periodic geodesic, contradicting the non-trapping assumption. 
\end{proof}

\noindent\textbf{High energy partition of unity} We first note that by \cite[Section 2]{VaZw}, the function $\phi(h^2 \Delta)$ is a semiclassical pseudodifferential operator with symbol $\phi(|\xi|^2)$. We now choose  smooth functions 
$q_i$ for $i=1,\dots,N+N'$ on ${^\SC T}^*\bbar{M}$ such that 
\begin{itemize}
\item $q_i \in C_0^\infty(\tilde B_i), \quad i = 1 \dots N$;
\item $q_{N+j} \in C_0^\infty(D_j), \quad j = 1 \dots N'$;
\item $\sum_{i=1}^{N+N'}q_i=1$  on $\{(m,\xi); |\xi|^2\in \supp \,  \phi\}$.
\end{itemize}
\begin{comment}
supported in $\{(m,\xi) \in{^\SC T}^*\bbar{M}; |\xi|_g\in[1-\delta,1+\delta]\}$, and so that 
\[\sum_{i=1}^{N+N'}q_i=1 , \textrm{ on }\{(m,\xi); |\xi|\in \supp \phi\}\]
where $\phi$ is as in Proposition \ref{outsidecone}
We also require that $q_i\in C_0^\infty(\til{B}_j)$ when $i=1,\dots,N$ and $q_{N+i}\in C_0^\infty(D_i)$ 
when $i=1,\dots,N'$. 
\end{comment}
Now we define the operators 
\begin{equation}\ilabel{defQ_j}
Q_i(h)={\rm Op}_h(q_i) \phi(h^2 \Delta), \,\, i=1,\dots,N+N',  
\end{equation}
where ${\rm Op}_h$ is a semiclassical quantization as explained before. By construction, we 
have 
\begin{equation}\label{partitionofId}
\sum_{i}Q_i(h)-\phi(h^2\Delta) = R(h) \in \Psi^{-\infty,\infty,-\infty}(M).
\end{equation}
Let us redefine $Q_1(h)$ to be $Q_1(h) + R(h)$, which does not change any microlocal properties of $Q_1(h)$. Then $(Q_i(h), (1 - \phi)(h^2 \Delta))$ is a partition of unity. 
Let us also observe that as the $Q_i(h)$ are uniformly bounded as operators $L^2 \to L^2$, and as they are Calder\'on-Zygmund operators in a uniform sense as $h \to 0$, then they are uniformly bounded as operators $L^p \to L^p$ for $1 < p < \infty$.
We shall frequently use the notation $Q_i$ below instead of $Q_i(h)$.
Since $\WF'(Q_i)\subset \til{B}_i$ and  $\WF'(Q_{i+N})\subset D_i$ for $i\leq N$,
in view of  Lemma \ref{partitioninx>eps} we have the following
 
\begin{lemma}\label{WFprop}
Let $\eps,\iota,N,N'$ and $B_i$ defined in Proposition \ref{partitioninx>eps}. For $1 \leq j,k\ \leq N+N'$,  one of the following alternatives is satisfied:\\
1) either $g^t(\WF'(Q_i))\cap \WF'(Q_j)=\emptyset$ for all $t\geq 0$\\
2) or $g^t(\WF'(Q_i))\cap \WF'(Q_j)=\emptyset$ for all $t\leq 0$\\
3) or $\WF'(Q_i)\cap \WF'(Q_j)\not= \emptyset$ with $\WF'(Q_i)\subset \{x<2\eps, \nu\in B_k\}$, $\WF'(Q_j)\subset \{x<2\eps, \nu\in B_{k'}\}$ for some $k,k'$ with $|k-k'|\leq 1$\\
4) or $\WF'(Q_i)\cap \WF'(Q_{j})\not= \emptyset$, $\WF'(Q_i)\cup\WF'(Q_{j})\subset \{x>\eps\}$ 
and ${\rm dist}(\WF'(Q_i),\WF'(Q_j))<\iota$.

Moreover, all but one of the $Q_i$ have operator wavefront set $\WF'(Q_i)$ disjoint from the outgoing radial set $\{ x = 0, \mu = 0, \nu = +1 \}$, and all but one of the $Q_i$ have operator wavefront set $\WF'(Q_i)$ disjoint from the incoming radial set $\{ x = 0, \mu = 0, \nu = -1 \}$. 
\end{lemma}

In the first case, one has $Q_j$ not outgoing related to $Q_i$, in the second case $Q_j$ is not incoming related to $Q_i$, while in the last two cases the microsupports are close one to each other.\\ 

\noindent\textbf{Low energy partition of unity} For a low energy partition of unity, we can effectively use the low energy partition of unity from \cite[Section 6]{GHS2}. For the reader's convenience we describe this here. 

We first remark that this partition of unity lives in the calculus of pseudodifferential operators $\Psi^0_k(M, \Omega_{k,b}^{1/2})$ defined in \cite{GH}, \cite{GHS1}. This calculus provides a way of defining pseudodifferential operators depending on an energy parameter $\lambda\in 
[0,\la_0]$ in a uniform way as  $\lambda \to 0$ (for some $\la_0>0$ fixed). In view of the choice of $Q_0$ below, we only describe this here for pseudodifferential operators supported in the set $\rho := x/\lambda \leq 2\epsilon$ for some small $\epsilon$. Then such pseudodifferential operators can be defined as follows: we see that an operator $A$, supported where $x, x' \leq 2\epsilon \lambda$, is in this calculus if for each boundary point $y_0 \in \pl\bbar{M}$ there is a conic-type neighbourhood $U$ with coordinates 
$z:U\to \rr^n$ so that $x=|z|^{-1}$
such that for all half-densities $u |dz|^{1/2}, u \in C_c^\infty(U)$ we have 
$$
(Au)(z)  = \Big( \big( \frac{\lambda}{2\pi} \big)^{n} \int \int e^{i\lambda(z-z') \cdot \zeta} a(\rho, \lambda, y, \zeta) u(z') \, d\zeta \, dz' \Big) |dz|^{1/2} , \quad z \in U,
$$
where $a$ is a zeroth order symbol in $\zeta$ depending smoothly on $(\rho, \lambda, y)$ and supported in $\rho \leq 2\epsilon$. Such operators are uniformly bounded on $L^2$; indeed, they are Calder\'on-Zygmund operators in a uniform sense as $\lambda \to 0$ and therefore uniformly bounded on $L^q$ as $\lambda \to 0$, $1 < q < \infty$. 
%We may suppose that the elements of our low energy partition of unity (apart from $Q_0$) have kernels supported where $\rho \leq 2\epsilon$, $\lambda d(z,z') \leq 1$. 
For our purposes here, we only need to consider operators of order $-\infty$. The wavefront set $\WF'(A)$ is then a subset of  ${^\SC T}_{\pl\bbar{M}}^*\bbar{M} \times [0, \lambda_0]$ and can be defined either by the vanishing properties of its symbol by analogy with the high-energy case, or equivalently in terms of its microlocal support as in \cite[Section 5]{GHS2}.

We first choose $Q_0$ to be multiplication by the function $1-\chi(\rho)$, where $\chi(\rho) = 1$ for $\rho \leq \epsilon/2$ and $\chi(\rho) = 0$ for $\rho \geq \epsilon$, for some sufficiently small $\epsilon$. Next, we choose $Q'_*$ such that $\Id - Q'_*$ is microlocally equal to the identity for $|\mu|_h^2 + \nu^2 \leq 3/2$, and microsupported in $|\mu|_h^2 + \nu^2 \leq 2$. Let $Q_* = \chi(\rho) Q_*'$. Finally, we write $\Id - Q_0 - Q_* = \chi(\rho) (\Id - Q'_*)$, which has compact microsupport, as a finite sum of operators $Q_i$, $1 \leq i \leq N$, where $Q_i$ is microsupported in $\tilde B_i$ from Lemma~\ref{partitioninx>eps}. 
Observe that the $Q_1, \dots, Q_N$ satisfy 

\begin{lemma}\ilabel{WFprop-lowenergy}
Let $\eps,N$ and $B_i$ be defined in Proposition \ref{partitioninx>eps}. Let $j,k\in [1,N]$, then one of the following alternative is satisfied:\\
1) either $\sup \{ \nu(q) \mid q \in \WF'(Q_i)\}  < \inf \{ \nu(q) \mid q \in \WF'(Q_j)\}$ \\
%$g^t(\WF'(Q_i))\cap \WF'(Q_j)=\emptyset$ for all $t\geq 0$\\
2) or $\inf \{ \nu(q) \mid q \in \WF'(Q_i)\}  >  \sup \{ \nu(q) \mid q \in \WF'(Q_j)\}$ \\
%2) or $g^t(\WF'(Q_i))\cap \WF'(Q_j)=\emptyset$ for all $t\leq 0$\\
3) or $\WF'(Q_i)\cap \WF'(Q_j)\not= \emptyset$ with $\WF'(Q_i)\subset \{x<\eps, \nu\in B_k\}$, $\WF'(Q_j)\subset \{x<\eps, \nu\in B_{k'}\}$ for some $k,k'$ with $|k-k'|\leq 1$.\\
Moreover, all but one of the $Q_i$ have operator wavefront set $\WF'(Q_i)$ disjoint from the outgoing radial set $\{ x = 0, \mu = 0, \nu = +1 \}$, and all but one of the $Q_i$ have operator wavefront set $\WF'(Q_i)$ disjoint from the incoming radial set $\{ x = 0, \mu = 0, \nu = -1 \}$. 
\end{lemma}

 %%%%%%%%%%%%%%%%%%%%%%%%%%%%%%%%
 %%%%%%%%%%%%%%%%%%%%%%%%%%%%%%%%

\section{Estimates close to the spectrum}
In this section, we prove estimates on the resolvent $(\Delta - \alpha)^{-1}$ for $|\arg \alpha| \leq \eta$, where we may take $\eta$ as small as we like thanks to Proposition~\ref{outsidecone}. In this section we use microlocal properties of the resolvent proved in \cite{GHS2}; these properties which we require here are recalled in Section~\ref{sec:hfe} so that no detailed knowledge of Legendre distributions, etc, is required to read this section. However, in the Appendix we give an alternative proof using positive commutator estimates, which avoids all use of Legendre distributions, which may be preferred by some readers. 

We first show that the desired estimate in a sector close to the spectrum are a consequence of the estimates on the boundary, thanks to 
Phragm\'en-Lindel\"of principle. We write it for $p=2n/(n+2)$, but the same proof obviously works for $p$ replaced by $q\in(1,\infty)$: the growth inside the sector is bounded by the growth on the boundary if that growth is a power of $|\alpha|$ (typically we shall take 
$|\alpha|^{n(\frac{1}{q}-\demi)-1}$). 

\begin{prop}\ilabel{PhrLind}
Let $\eta\in(0,\pi/2)$ and $p\in(1,\infty)$. Assume that we have estimates 
$$
\big\| (\Delta - \alpha)^{-1} \big\|_{L^p \to L^{p'}} \leq C, \quad \arg \alpha = \eta, \ \Re \alpha > 0
$$
on the ray $\arg \alpha = \eta$, and estimates on the spectrum:
$$
\big\| (\Delta - (\alpha + i0))^{-1} \big\|_{L^p \to L^{p'}} \leq C, \quad  \alpha > 0.
$$
Then  there exists $C>0$ such that for all $\alpha\in \cc$ satisfying  
$0 \leq \arg(\alpha)\leq \eta$, one has 
\[||(\Delta-\alpha)^{-1}||_{L^p\to L^{p'}}\leq C.\]
\end{prop} 
\begin{proof}
Let $\pi/2>\eta>0$ be fixed and let $\varphi,\psi\in C_0^\infty(M)$ with $||\varphi||_{L^p}\leq 1$ and $||\psi||_{L^p}\leq 1$. Then we define the function
\[F(\alpha):=\cjg (\Delta-\alpha)^{-1}\varphi,\psi\cjd \]
which is holomorphic in the sector $0<\arg(\alpha)<\eta$. 
%Combining \eqref{disttospecleq1} and Lemma \ref{reallowfr}, 
By assumption,  there is $C>0$ such that
\[ |F(\alpha)|\leq C||\varphi||_{L^p}||\psi||_{L^p}\]
for all $\alpha>0$ and for all $\alpha\in \rr_+e^{i\eta}$.
The function $F(\alpha)$ is continuous in $U_\eta:=\{\alpha\in\cc; 0\leq \arg(\alpha)\leq \eta, 0<|\alpha|\}$ 
as follows from \cite[Prop. 14]{Mel}.  
Moreover, \cite[Prop 1.27]{RoTa} shows that there is $C'>0$ such that for $\alpha \in U_\eta, |\alpha|\leq 1$, we have 
\[ \begin{gathered}
||x^{1/2+\eps}R(\alpha)x^{1/2+\eps}||_{L^2\to L^2}\leq C'/|\alpha|^{L} 
\end{gathered}\]  
for some $L\geq 0$, while this follows with $L = 1$ for $\alpha \in U_\eta, |\alpha| \geq 1$ by \cite{VaZw}. Thus $|F(\alpha)|\leq C'' \max(|\alpha|^{-L}, |\alpha|^{-1})$ in $U_\eta$ for some $C''>0$ depending on $\varphi,\psi$.
We can apply Phragmen-Lindel\"of for $F(e^{z})$ in the half strip $\{|{\rm Im}(z)|\leq \eta, {\rm Re}(z)\leq 0\}$  and we deduce
that $|F(\alpha)|\leq C||\varphi||_{L^p}||\psi||_{L^p}$ in the closure of $U_\eta$. 
\end{proof}  
This reduces our analysis to estimates on the real line. 

\subsection{The high frequency estimates}\ilabel{sec:hfe}
The goal in this subsection is to prove the 
\begin{prop}\ilabel{hightoprove}
Let $(M,g)$ be an asymptotically conic manifold with non-trapping geodesic flow. Let $A,\eta>0$, 
then for $p=\frac{2n}{n+2}$ and $1/p'+1/p=1$, 
there exists $C>0$ such that for all $\alpha\in \cc$ with $|\alpha|>A$ 
and $\pm\arg(\alpha)\in (0,\eta)$,
\[||\phi(\Delta/|\alpha|)(\Delta-\alpha)^{-1}||_{L^p\to L^{p'}}\leq C.\]
\end{prop}
\begin{remark}\ilabel{endpointhigh} As it will be clear from its proof, the estimate in this Proposition also holds with $p$ replaced by any 
$q\in [\frac{2n}{n+2},\frac{2(n+1)}{n+3}]$: there is $C>0$ such that 
\[||\phi(\Delta/|\alpha|)(\Delta-\alpha)^{-1}||_{L^q\to L^{q'}}\leq C |\alpha|^{n(\frac{1}{q}-\demi)-1}\]
for all $\alpha\in \cc$ with $|\alpha|> A>0$ and $\pm\arg(\alpha)\in (0,\eta)$. We do not discuss the details, this is straightforward.
\end{remark}

We turn the problem into a semiclassical problem by setting $h=1/\sqrt{|\alpha|}$.
Using the microlocal partition of unity $\sum_iQ_i$  
introduced in \eqref{defQ_j}, we see that it suffices to prove for the high-frequency part that 
for $h_0>0$ fixed there is $C>0$ such that 
\begin{equation}\ilabel{toprovelocal}
||Q_i\phi(h^2\Delta)(h^2\Delta-\beta-i\gamma)^{-1}Q_j||_{L^p\to L^{p'}}\leq Ch^{-2}
\end{equation}
for all $\beta\in (1-\delta,1+\delta)$, all $\gamma\in(0,\eta)$, all $h\in (0,h_0)$, and $i,j\in \{1,\dots,N+N'\}$.\\

\noindent\textbf{The `near diagonal' estimate.} In Theorem 1.12 in  \cite[Th. 1.12]{GHS2}, we proved  that 
for non-trapping asymptotically conic manifolds, the following holds:
\begin{prop}
Let $N,\iota$ and $Q_i=Q_i(h)$ as in Lemma \ref{partitioninx>eps}. If $N$ is chosen large enough 
and $\iota$ small enough, then  for all $\ell\in\nn$ and all $i,j$ with $Q_i,Q_j$ satisfying either 3) or 4) 
in Lemma \ref{WFprop}, there is $C>0$ such that for all $h\in (0,h_0)$ and 
all $\la\in [\frac{1-\delta}{h}, \frac{1+\delta}{h}]$
\[|Q_i\pl^j_{\la}dE_{\sqrt{\Delta}}(\la)Q_j (z,z')|\leq C\la^{n-1-\ell}(1+\la d(z,z'))^{-\frac{n-1}{2}+\ell}.\]
\end{prop} 

Now, by Lemma \ref{uniformonspect} and the remark that follows, we obtain the following corollary. 
\begin{coro}\label{cor1}
Let $N,\iota$ as in Lemma \ref{partitioninx>eps}, let $p=\frac{2n}{n+2}$ and
$1/p'+1/p=1$. If $N$ is chosen large enough and $\iota$ small enough, 
then  for all $j,k$ with  $Q_i,Q_j$ satisfying either 3) or 4) 
in Lemma \ref{WFprop}, there is $C>0$ such that for 
all $\beta\in (1-\delta,1+\delta)$ and $\gamma\in(0,\eta)$
\[ ||Q_i\phi(h^2\Delta)(h^2\Delta-\beta-i\gamma)^{-1}Q_j||_{L^p\to L^{p'}}\leq Ch^{-2}.\]
\end{coro}
Using  Remark \ref{2(n+1)/(n+3)} after Lemma \ref{uniformonspect} the same estimate with $\mc{O}(h^{-2n(\frac{1}{q}-\demi)})$ bound holds for $L^q\to L^{q'}$ norms when $q\inÊ[\frac{2n}{n+2},\frac{2(n+1)}{n+3}]$.\\

\noindent\textbf{The `off diagonal' estimates.}
For the `off diagonal' estimates, that is, estimates when $(Q_i, Q_j)$ satisfy 1) or 2) in Lemma \ref{WFprop}, we will obtain the estimates in the case when we are on the spectrum, i.e. to $Q_i (h^2 \Delta - (\beta \pm i0))^{-1} Q_j$. 
We use the description of the resolvent kernel as a Legendre distribution as in \cite{GHS1}; this was based on the high energy construction of the resolvent from \cite{HW}. We first recall relevant properties of Legendre distributions needed in our argument.  Full details are in the papers \cite{MZ, HW, GHS1}. \\

%Let us first consider the high energy case, $\lambda \geq \lambda_0$. In this case we write $h = 1/\lambda$ which we consider a semiclassical parameter. 
Legendre distributions, depending on the semiclassical parameter $h$ are best understood on the space $M^2_b \times [0, h_0]$, $h_0 = \lambda_0^{-1}$, where $M^2_b$ is the b-double space of $M$, defined to be the radial blowup of $\bbar{M}^2$ at the corner $(\partial\bbar{M})^2$. 
 The manifold $M^2_b$ is a smooth manifold with corners obtained by replacing the corner $(\pl \bbar{M})^2$ in $\bbar{M}^2$
by the interior pointing normal unit bundle at the corner, identified with $[0,1]\x (\pl\bbar{M})^2$, 
we refer to \cite{Melb} for details on the construction of $M^2_b$. Informally, $M^2_b$ is the resolved space so that $x/(x'+x), x'/(x+x'),x+x'$ are smooth functions near the corner $x=x'=0$. There are $3$ boundary hypersurfaces in 
$M^2_b$:
\[\lb:=\{\frac{x}{x'+x}=0\}, \quad \rb=\{\frac{x'}{x+x'}=0\},\quad \bfc:=\{x+x'=0\}.\]
The boundary  face $\bfc$ has interior which can be identified to $(0,\infty)\x(\pl\bbar{M})^2$ by restricting 
the coordinates $(s:=x/x',y,y')$ to $\bfc$.
The natural semiclassical phase space on $M^2_b$ is the tensor product of the two vector bundles obtained by lifting the scattering cotangent bundle over $\bbar{M}$ to $M^2_b$ via either the left or right stretched projections. Thus, in the interior of $M^2_b$ it just looks like the usual cotangent space, with spatial coordinates $(z, z')$, $z \in M$ and dual coordinates $(\zeta, \zeta')$,  but near the blown-up face we would use spatial coordinates $(x/x', x', y, y')$ where $x/x' \leq C$ or $(x, x'/x, y, y')$ where $x'/x \leq C$, and fibre coordinates $(\nu, \mu, \nu', \mu')$. 

We then introduce the product space $M^2_b \times [0, h_0]$ and adopt the usual semiclassical scaling. That is, we consider the semiclassical vector fields $h \partial_{z_i}$ in $M$, or $h x^2\partial_x$, $h x \partial_{y_i}$ as our basic building blocks (since the semiclassical Laplacian $h^2 \Delta$ is an elliptic combination of such vector fields), for which dual vector fields are 
$dx/(x^2h)$ and $dy_i/(xh)$. Hence we write covectors on $M^2_b \times [0, h_0]$ in the form
\begin{equation} \begin{gathered}
\zeta \cdot \frac{dz}{h} + \zeta' \cdot \frac{dz'}{h} + T d\big( \frac1{h} \big) \text{ in the interior of } M^2_b \times [0, h_0], \text{ or} \\
\nu \frac1{h} d\big( \frac1{x} \big) + \sum_i \mu_i \frac{dy_i}{xh} + 
\nu' \frac1{h} d\big( \frac1{x'} \big) + \sum_i \mu_i \frac{dy'_i}{x'h} + T d\big( \frac1{h} \big) \text{ near } \bfc \times [0, h_0]. 
\end{gathered}\end{equation}
defining linear coordinates $(\nu, \mu, \nu', \mu', T)$ on the fibres on $M^2_b \times [0, h_0]$. 

A (semiclassical) Legendre distribution $F$ on $M^2_b$ is a kernel whose microlocal properties are determined by a Legendre submanifold associated to the `main face' $\mf := M^2_b \times \{ 0 \}$. In fact, it turns out that the restriction of the phase space to $\{ h = 0 \}$ is a contact manifold in a natural way, with contact structure given in local coordinates over the interior of $\mf$ by $-dT + \zeta \cdot dz$. Then a Legendre submanifold of this space can equivalently (by forgetting the $T$ coordinate) be thought of as a Lagrangian submanifold of $T^* M^2$ and, in terms of this Lagrangian submanifold, a Legendre distribution is precisely a semiclassical Lagrangian distribution associated to this Lagrangian submanifold. 

The Legendre submanifold has a continuous extension to the boundary hypersurfaces of the phase space over $M^2_b \times [0, h_0]$ lying over $\bfc$, the left boundary $\lb$ and the right boundary $\rb$. In this paper we will focus on the \emph{microlocal support} of the Legendre distribution, which has components at $\mf$, at $\bfc \times [0, h_0]$, and at $\lb\times [0, h_0]$ and $\rb\times [0, h_0]$; these will be denoted $\WF'_{\mf}(F)$, $\WF'_{\bfc}(F)$, $\WF'_{\lb}(F)$, $\WF'_{\rb}(F)$, and will be described in the following paragraphs. 
In particular, near a point $m\in M^2_b\x[0,h_0]$ with $m\in {\rm f}$ for some ${\rm f}\in\{\rb,\lb,\bfc,\mf\}$, a Legendre distribution $F$ vanishes to infinite order at the boundary face ${\rm f}$ if $m$ does not belong to the projection 
of $\WF'_{\rm f}(F)$ to the base ${\rm f}\x[0,h_0]$.

In the case of $\WF'_{\mf}(F)$ this is obtained from the Legendre submanifold by negating the right cotangent variables (the same way that the canonical relation is obtained from the Lagrangian submanifold for FIOs). It is shown in \cite{GHS2} that for $F$ the kernel of the outgoing or incoming resolvent  $(h^2 \Delta - (\beta \pm i0))^{-1}$, the microlocal support at $\mf$ consists of a diagonal (or pseudodifferential) part together with the forward geodesic flow relation on $M$:
\begin{equation}\begin{gathered}
\WF'_{\mf}\big((h^2 \Delta - (\beta \pm i0))^{-1} \big) = \big\{ (z, \zeta, z', \zeta', T) \mid z=z', \zeta = \zeta', T = 0 \big\}  \\
\cup \ L'_\pm := \big\{ (z, \zeta, z', \zeta', T) \mid |\zeta|^2_g = \beta ; (z, \zeta) \text{ lies on the forward ($+$)/backward ($-$)  } \\ 
\text{ geodesic ray starting from } (z', \zeta), \ \pm T \text{ is the (nonnegative) length of this geodesic. }  \big\}
\ilabel{WFmf}\end{gathered}\end{equation}
At $\bfc \times [0, h_0]$, whose interior can be viewed as $(0,\infty)_{x/x'}\x \pl\bbar{M}\x\pl\bbar{M}\times [0, h_0]$, the microlocal support is  the diagonal relation, together with the forward/backward geodesic flow relation on the exact metric cone 
$C(\partial\bbar{M},h(0)):=(\rr^+_r\x \pl\bbar{M}; dr^2+r^2h(0))$, together with a purely incoming/outgoing set:
\begin{equation}\begin{gathered}
\WF'_{\bfc}\big((h^2 \Delta - (\beta \pm i0))^{-1}\big) = \big\{ (x/x', y, y', \nu, \mu, \nu', \mu', h) \mid x/x' = 1, \ y = y', \ \nu = \nu', \ \mu = \mu' \big\} \\
\cup \ 
(L^{\bfc})'_\pm := \big\{ (x/x', y, y', \nu, \mu, \nu', \mu', h) \mid \nu^2 + |\mu|_h^2 = \beta, \text{ and there exist $(r,r')$ such that } \\
x/x' = r'/r, \text{ and $(r, y, \nu, \mu)$ lies on the forward ($+$)/backward ($-$) } \\
\text{ geodesic ray starting from  $(r', y', \nu', \mu')$ on the cone  $C(\partial\bbar{M}, h(0))$}  \big\}  \\
\cup \ 
 \big\{ (x/x', y, y', \nu, \mu, \nu', \mu', h) \mid \nu = \pm \sqrt{\beta}, \nu' = \mp \sqrt{\beta}, \mu = \mu' = 0 \big\}
\end{gathered}\ilabel{WFbf}\end{equation}
(note that on an exact cone, there is a dilation invariance which means that only the ratio between the two radial variables $r/r'$, or equivalently only the ratio $x/x'$, is relevant).  
By \eqref{nudot}, the variable $\nu$ is monotone along the geodesic. Consequently, we have 
\begin{equation} 
\nu \geq \nu' \text{ on } (L^\bfc)'_+, \quad 
\nu \leq \nu' \text{ on } (L^\bfc)'_-,
\ilabel{nunu'}\end{equation}
with strict inequalities away from the diagonal. 

Finally, the microlocal supports at the left and right boundaries are subsets of ${^\SC T}_{\pl\bbar{M}}^*\bbar{M} \times [0, h_0]$ and for the resolvents are given by 
\begin{equation}\begin{gathered}
\WF'_{\lb}\big((h^2 \Delta - (\beta \pm i0))^{-1}\big) = \big\{ x = 0, \mu= 0, \nu = \pm \sqrt{\beta} \big\}, \\
\WF'_{\rb}\big((h^2 \Delta - (\beta \pm i0))^{-1}\big) = \big\{ x' = 0, \mu'= 0, \nu' = \mp \sqrt{\beta} \big\}.
\end{gathered}\ilabel{WFlbrb}\end{equation}

For the purposes of this paper, we need to know two properties of the microlocal support. The first is how it behaves under composition with pseudodifferential operators. The following result was proved in \cite[Section 7]{GHS2}:

\begin{prop}\ilabel{msprop} Let $F$ be a Legendre distribution/intersecting Legendre distribution/Legendre conic pair on $M^2_b \times [0,h_0]$. Let $Q, Q' \in \Psi^{-\infty, l, k}(M)$ be pseudodifferential operators such that their operator wavefront set $\WF'(Q), \WF'(Q')$ is compact in ${^\SC T}^*\bbar{M}$.  Then the microlocal support of $Q F Q'$ satisfies
\begin{equation*}\begin{gathered}
\WF'_{\mf}(QFQ') \subset \WF'_{\mf}(F) \cap \big\{ (z, \zeta, z', \zeta', T) \mid  (z, \zeta) \in \WF'(Q), \ (z', \zeta') \in \WF'(Q') \big\}  ;
\\
\WF'_{\bfc}(QFQ') \subset \WF'_{\bfc}(F) \cap \big\{ (x/x', y, y', \nu, \mu, \nu', \mu', h) \mid  (y, \nu, \mu, h) \in \WF'(Q), \\ 
\phantom{aaaaaaaaaaaaaaaaaaaaaaaaaaaaa}\ (y', \nu', \mu', h) \in \WF'(Q') \big\} ; \\
\WF'_{\lb}(QFQ') \subset \WF'_{\lb}(F) \cap \WF'(Q) ; \\
\WF'_{\rb}(QFQ') \subset \WF'_{\rb}(F) \cap \WF'(Q').
\end{gathered}\end{equation*}
\end{prop}

The second fact about the microlocal wavefront set we need is that a Legendre distribution $F$ has trivial kernel, i.e. kernel in $h^\infty x^\infty {x'}^\infty C^\infty(\bbar{M}^2 \times [0, h_0])$ if and only if its microlocal support is empty. 

\begin{comment}
We are now ready to describe what happens when the resolvent is composed with (semiclassical) scattering pseudodifferential operators $Q_i=Q_i(h)$. From \cite[Lemma 7.1]{GHS2} we see that $Q_i (h^2 \Delta - (\beta \pm i0))^{-1} Q_j$ is a Legendre distribution associated to the same Legendrian relation, but with microsupport restricted so that on the main face $\mf$, we  have $(z, \zeta, z', \zeta', T)$ in the Legendrian relation of the composition only if $(z, \zeta)$ is in $\WF'(Q_i)$ and $(z', \zeta')$ is in $\WF'(Q_j)$. Similarly, at $\bfc$, we only have $(x/x', y, y', \nu, \mu, \nu', \mu')$ in the Legendrian relation of the composition only if $(y, \tau, \mu)$ is in $\WF'(Q_i)$ and $(y', \nu', \mu')$ is in $\WF'(Q_j)$. Finally, at the left boundary, the point $(y, \nu, \mu)$ survives in the Legendre relation only if this point is in $\WF'(Q_i)$ and at the right boundary,  the point $(y', \nu', \mu')$ survives in the Legendre relation only if this point is in  $\WF'(Q_j)$. 
%In this way, we can use a partition of unity to localize the Legendrian relation as much as we like. 
\end{comment}
From these statements, it is  straightforward to obtain the following result. 

\begin{prop}\ilabel{microlocaltrivial}
Let $h_0>0$, let $i,j$ so that $Q_i,Q_j$ satisfying 2) in Lemma \ref{WFprop}, ie. $Q_i$ is not outgoing-related to $Q_j$, 
then for all $L>0$, all $q,q'\in (1,\infty)$ there is $C>0$ such that for 
all $\beta\in (1-\delta,1+\delta)$ and $h\in(0,h_0)$
\[ ||Q_i(h^2\Delta-\beta-i0)^{-1}Q_j||_{L^q\to L^{q'}}\leq Ch^{L}.\]
Similarly, if $Q_j$ not outgoing related to $Q_i$, then 
\[ ||Q_i(h^2\Delta-\beta+i0)^{-1}Q_j||_{L^q\to L^{q'}}\leq Ch^{L}.\]
\end{prop}

\begin{proof} For definiteness, we assume that $Q_i,Q_j$ satisfy 2) in Lemma \ref{WFprop}, and we prove the first estimate in the Proposition. The other is obtained in exactly the same way. 

We begin by recalling that by \cite[Theorem 1.1]{HW}, the incoming and outgoing resolvents are a sum of an intersecting Legendre distribution and a Legendre conic pair, so Proposition~\ref{msprop} applies. Property  2) in Lemma \ref{WFprop} together with \eqref{WFmf} and  the first line of Proposition~\ref{msprop} implies that $\WF'_{\mf}(Q_i(h^2\Delta-(\beta+i0))^{-1}Q_j)$ is empty. Similarly,   property 2),  \eqref{WFbf}, \eqref{nunu'}  and the second line of Proposition~\ref{msprop} shows that $\WF'_{\bfc}(Q_i(h^2\Delta-(\beta+i0))^{-1}Q_j)$ is empty. Finally, given the last statement in Lemma~\ref{WFprop}, we see that $Q_i$ has operator wavefront set disjoint from the outgoing radial set, and $Q_j$ has operator wavefront set disjoint from the incoming radial set (otherwise property 2) could not be satisfied). Hence, using \eqref{WFlbrb} and the last two lines of Proposition~\ref{msprop}, $\WF'_{\lb}(Q_i (h^2\Delta-(\beta+i0))^{-1}Q_j)$ and $\WF'_{\rb}(Q_i (h^2\Delta-(\beta+i0))^{-1}Q_j)$ are empty.

It follows that the kernel of $Q_i(h^2\Delta-(\beta+i0))^{-1}Q_j$ is  in $h^\infty x^\infty {x'}^\infty C^\infty(\MMb \times [0, h_0])$. The estimate now follows trivially. 
\end{proof}

Finally, we obtain 
\begin{coro}\ilabel{cor2}
Let $N,\iota$ as in Lemma \ref{partitioninx>eps} and let $\phi$ be as in Section~\ref{sec:outside-sector}. If $N$ is chosen large enough and $\iota$ small enough, 
then  for all $j,k$ with  $Q_i,Q_j$ satisfying either 1) or 2) 
in Lemma \ref{WFprop}, there is $C>0$ such that for $p=\frac{2n}{n+2}$ with $1/p+1/p'=1$
and all $\beta\in (1-\delta,1+\delta)$, 
\[ ||Q_i\phi(h^2\Delta)(h^2\Delta-\beta\pm i0)^{-1}Q_j||_{L^p\to L^{p'}}\leq Ch^{-2}.\]
\end{coro}
\begin{proof} Since $Q_i$ and $Q_j$ are zeroth order pseudodifferential operators, they are bounded on $L^p$ uniformly in $h$. Hence, using \eqref{disttospecleq1}, we see that 
\begin{equation}
\Big\| Q_i (1 - \phi)(h^2 \Delta) (h^2 \Delta - (\beta \pm i0))^{-1} Q_j \Big\|_{L^p \to L^{p'}} \leq Ch^{-2}. 
\ilabel{222}\end{equation}
Then, if $(Q_i, Q_j)$ satisfy 2), we see from Proposition~\ref{microlocaltrivial} and \eqref{222} that 
\begin{equation}
\Big\| Q_i \phi(h^2 \Delta) (h^2 \Delta - (\beta + i0))^{-1} Q_j \Big\|_{L^p \to L^{p'}}\leq C h^{-2}. 
\ilabel{333}\end{equation}
Now using the uniform boundedness of $Q_i$, $Q_j$ on $L^q$ spaces again, Corollary~\ref{cor:diff}, and Proposition~\ref{microlocaltrivial}, we see that 
\begin{equation}
\Big\| Q_i \phi(h^2 \Delta) \Big( (h^2 \Delta - (\beta + i0))^{-1} - (h^2 \Delta - (\beta - i0))^{-1} \Big) Q_j \Big\|_{L^p \to L^{p'}} \leq C 
h^{-2}. 
\ilabel{444}\end{equation}
Combining \eqref{333} and \eqref{444} we see that we also have 
\begin{equation}
\Big\| Q_i \phi(h^2 \Delta) (h^2 \Delta - (\beta - i0))^{-1} Q_j \Big\|_{L^p \to L^{p'}} \leq C h^{-2}
\ilabel{555}\end{equation}
(where we have the incoming resolvent instead of the outgoing resolvent as in \eqref{333}). 
The same argument (with incoming and outgoing reversed) works for $Q_i,Q_j$ satisfying 1). 
\end{proof}

\begin{proof}[Proof of Proposition \ref{hightoprove}]
Now Corollary \ref{cor1} and \ref{cor2} together prove \eqref{toprovelocal} which in turn proves Proposition \ref{hightoprove} 
\end{proof}

\subsection{Low frequency estimates} 
Here we use the low frequency partition of unity from Section~\ref{micpart}, and deduce estimates for the $L^p$ to $L^{p'}$ norm of $(\Delta - \alpha)^{-1}$ for $\Re \alpha \leq C$.\\ 

\noindent\textbf{The `near diagonal' estimate.} 
The estimate \eqref{assumption} is proved in \cite[Th. 1.12]{GHS2}, and Lemma~\ref{uniformonspect} finishes the proof for the near diagonal terms $Q_0 (\Delta - \alpha)^{-1} Q_0$, $Q_* (\Delta - \alpha)^{-1} Q_*$ and $Q_i (\Delta - \alpha)^{-1} Q_j$ when $|i-j| \leq 1$.\\

\noindent\textbf{The `off diagonal' estimate.} 
To prove the low energy off diagonal estimates, we begin by noting that thanks to  Proposition~\ref{PhrLind} it is only necessary to prove the estimate on the spectrum, that is, for $(\Delta-\beta \pm i0)^{-1}$ for  $0 < \beta < C$; we often write $\beta = \lambda^2$, where $\lambda > 0$.  Here we use the microlocal structure of the low energy resolvent as proved in \cite{GHS1}; the argument is entirely analogous to the argument in the high energy setting, with the main difference being that the low energy resolvent has polyhomogeneous expansions 
as $\lambda \to 0$, as opposed to the high energy resolvent which is oscillatory as $h \to 0$. Let us recall that the low energy space introduced in \cite{GH} and used in \cite{GHS1, GHS2} is a blown up version of $\bbar{M}^2 \times [0, \lambda_0]$. The space is obtained by first blowing up the codimension-$3$ corner $\pl\bbar{M}^2\x\{0\}$ and then the three codimension-$2$ corners, corresponding
to $\pl \bbar{M}\x M\x\{0\}$, $M\x \pl \bbar{M}\x\{0\}$ and $\pl \bbar{M}\x \pl\bbar{M}\x(0,\lambda_0]$.
There are four boundary hypersurfaces at $\lambda = 0$, labelled $\zf, \lbo, \rbo$ and $\bfo$ according as they arise from $M \times M \times \{ 0 \}$, $\partial \bbar{M} \times M \times \{ 0 \}$, $M \times \partial \bbar{M} \times \{ 0 \}$, or $\partial \bbar{M} \times \partial \bbar{M} \times \{ 0 \}$. The other boundary hypersurfaces are $\lb$, $\rb$ and $\bfc$ which arise from $\partial \bbar{M} \times M \times [0, \lambda_0]$, $M \times \partial \bbar{M} \times [0, \lambda_0]$ and $\partial \bbar{M} \times \partial \bbar{M} \times [0, \lambda_0]$. For the low energy resolvent, we need to keep track of the microlocal support, which lives at $\bfc$, $\lb$ and $\rb$ only and is given by \eqref{WFbf}, \eqref{WFlbrb}, together with the order of vanishing  at $\lambda = 0$, i.e. at $\zf, \lbo, \rbo$ and $\bfo$. We have the analogue of Proposition~\ref{msprop} for zeroth order pseudodifferential operators $Q, Q'$ in the low energy case (see \cite[Section 6]{GHS2}), together with the fact that $QFQ'$ is conormal at each boundary hypersurface at $\lambda = 0$, with the same order of vanishing there as $F$.

\begin{prop}\ilabel{microlocaltrivial-le}
Let $(Q_0, Q_*, Q_1, \dots, Q_N)$ denote the low energy partition of unity constructed in Section~\ref{sec:mpou}. Suppose that $Q_i,Q_j$ satisfy 2) in Lemma \ref{WFprop-lowenergy}, ie. $Q_i$ is not outgoing-related to $Q_j$. Then there is $C>0$ such that for 
all $0 < \beta=\la^2 < C$, 
\[ ||Q_i(\lambda)(\Delta-(\beta+i0))^{-1}Q_j(\lambda)||_{L^p\to L^{p'}}\leq C.\]
Similarly, if $Q_j$ not outgoing related to $Q_i$, then 
\[ ||Q_i(\lambda)(\Delta-(\beta-i0))^{-1}Q_j(\lambda)||_{L^p\to L^{p'}}\leq C.\]
\end{prop}

\begin{remark}\ilabel{endpointlow} We can replace $p$ by any $q \in [1, 2n/(n+1))$ here, and then the right hand side becomes $|\beta|^{n(1/q - 1/2) - 1}$.
\end{remark}

\begin{proof}
In this proof, it will be convenient to use the following terminology: we will call a kernel `acceptable' if it is bounded by a constant times 
\begin{equation}
\lambda^{n-2} \big( 1 + \frac{\lambda}{x} \big)^{-(n-1)/2} \big( 1 + \frac{\lambda}{x'} \big)^{-(n-1)/2}.
\ilabel{acc}\end{equation}
A straightforward computation shows that this kernel has uniformly bounded $L^{p'}$ norm on $M \times M$ (recall that the measure looks like $x^{-(n+1)}dx \, dy$ in each factor of $M$, for small $x$). Therefore, any acceptable kernel is  uniformly  bounded as a map $L^p(M) \to L^{p'}(M)$.  More generally, as a map from $L^q(M)$ to $L^{q'}(M)$ with $q\in[\frac{2n}{n+2},\frac{2(n+1)}{n+3}]$,
and $1/q+1/q'=1$, the norm is $\mc{O}(\lambda^{2n(\frac{1}{q}-\demi)-2})$.

The resolvent kernel vanishes to order $0$ at the $\zf$ face and order $n-2$ at all the other boundary hypersurfaces at $\lambda = 0$ \cite[Theorem 3.9]{GHS1} (note that in \cite{GHS1}, we have $\nu_0 = n/2 - 1$; also note that the density half-bundle in Theorem 3.9 differs from the Riemannian density bundle by factors of $\rho_{\lbo}^{n/2} \rho_{\rbo}^{n/2} \rho_{\bfo}^n$). Moreover it vanishes to order $(n-1)/2$ at the left and right boundaries.  Note that in our partition $Q (\Delta - \beta \pm i0)^{-1} Q'$, only the term $Q_0 (\Delta - \beta \pm i0)^{-1} Q_0$ has support meeting the $\zf$ face and is a `near-diagonal' term, so in analyzing the off-diagonal terms we can ignore the vanishing order at this face; all off-diagonal terms are $\mc{O}(\lambda^{n-2})$ at $\lambda = 0$. 

Consider first the compositions  $Q_0 (\Delta-\beta \pm i0)^{-1}Q'$ and $Q'(\Delta-\beta \pm i0)^{-1}Q_0$, where $Q' \neq Q_0$. These kernels  vanish in a neighbourhood of $\bfc$. Thus, we are only left with the expansions at the left or right boundaries, together with the hypersurfaces at $\lambda = 0$ excluding $\zf$. Observe that $(1 + \lambda/x)^{-1}$ is a product of boundary defining functions for $\bfc$ and the left boundary, while $(1 + \lambda/x')^{-1}$ is a product of boundary defining functions for $\bfc$ and the right boundary. Since the kernels $Q_0 (\Delta-\beta \pm i0)^{-1}$ and $(\Delta-\beta \pm i0)^{-1}Q_0$  vanish to order $(n-1)/2$ at the left and right boundaries, to order $\lambda^{n-2}$ at $\lambda = 0$ and to infinite order at $\bfc$, they are bounded by a constant times \eqref{acc} and hence acceptable. 
Also, as noted in \cite{GHS2}, the operators $Q_*$ and $Q_i$ are bounded on $L^q$ uniformly in $\lambda$, for $1 < q < \infty$, so we see that all terms  $Q_0 (\Delta-\beta \pm i0)^{-1} Q_j$ and $Q_j(\Delta-\beta \pm i0)^{-1}Q_0$, $j \in \{ *, 1, 2, \dots, N \}$, are uniformly bounded from $L^p$ to $L^{p'}$. 

Next consider the terms $Q_* (\Delta-\beta \pm i0)^{-1}Q'$ and $Q'(\Delta-\beta \pm i0)^{-1}Q_*$. We use 
Proposition~\ref{msprop} to see that this form of sandwiching of the resolvent wipes out the piece $(L^{\bfc})'_{\pm}$ of the microlocal support completely, so the microsupport is contained in the diagonal  (only possible for $Q' = Q_0$) together with the left or right boundaries. Correspondingly, this term is the sum of a pseudodifferential operator in the calculus $\Psi^0_k(M, \Omega_{k,b}^{1/2})$ plus an acceptable term, both of which are uniformly bounded on $L^p$. 

Finally we consider terms of the form $Q_i (\Delta-\beta + i0)^{-1} Q_j$ where $i, j \geq 1$ and where  $Q_i$ is not outgoing-related to $Q_j$.  Using Proposition~\ref{msprop}, this  sandwiching of the resolvent again wipes out the piece $(L^{\bfc})'_{\pm}$ of the microlocal support completely, so again the microsupport is contained in the left or right boundaries.  Hence these terms are acceptable and uniformly bounded on $L^p$.   
\end{proof}

The low-energy estimates are completed by proving the analogue of Corollary~\ref{cor2} in the low-energy setting. Since the argument is identical to the high-energy case, we omit the details.

%%%%%%%%%%%%%%%%%%%%%%%%%%%%%%%
%%%%%%%%%%%%%%%%%%%%%%%%%%%%%%%
%%%%%%%%%%%%%%%%%%%%%%%%%%%%%%%
\appendix

\section{A positive commutator approach for the off-diagonal estimate}\ilabel{sec:alternative}
In this section we outline an alternative approach to the off-diagonal estimates in Proposition~\ref{microlocaltrivial}, which does not use the Legendre structure of the spectral measure. Instead, we use positive commutator estimates in the spirit of \cite{Mel} and especially \cite{VaZw}. In fact, our estimates are essentially  localized versions of the following global commutator estimate from  \cite{VaZw}:

\begin{prop}{\bf [Vasy-Zworski]} \ilabel{VaZw}
Let $(M,g)$ be a non-trapping asymptotically conic manifold. Let $h_0>0$, then for each $\eps>0$ small, 
there exists $C>0$ such that for all $h\in(0,h_0)$, all $\gamma>0$ small and all $f\in x^{1/2 + \eps}L^2(M)$
\begin{equation}
\big\| \big( h^2 \Delta - 1 \pm i\gamma\big)^{-1} f \big\|_{x^{-1/2 - \epsilon} L^2(M)} \leq C h^{-1} \| f \|_{x^{1/2 + \epsilon}L^2(M)}.
\ilabel{VZest}\end{equation} 
\end{prop}

We give two lemmas, the first of which is needed in the proof of the second. The desired off-diagonal estimates follow immediately from the second lemma. Although we want to use them for non-trapping metrics, 
we state them in a general setting where the geodesic flow can have trapped trajectories, since we believe that 
these Lemmas could be useful for related problems in trapping situations. Lemma \ref{lem:poscomm2} below can be seen as a propagation of 
singularities result. We recall that the \emph{forward trapped set} $\Gamma_+$
(resp. \emph{backward trapped set} $\Gamma_-$) is the closure of of the set of points $\zeta\in T^*M$ such that $g^t(\zeta)$ belongs
to a compact set for all $t\geq 0$ (resp. all $t\leq 0$).

\begin{lemma}\ilabel{lem:poscomm1} 
Let $(M,g)$ be an aymptotically conic manifold and assume that there exists $J\geq 1$ such that 
for all small $\eps>0$ there is $C>0, h_0>0$ such that for all $h\in(0,h_0)$, all small $\gamma>0$ 
and all $f\in x^{-1/2 - \epsilon} L^2(M)$
\begin{equation}\ilabel{assume}
 \big\| \big( h^2 \Delta - 1 \pm i\gamma\big)^{-1} f \big\|_{x^{-1/2 - \epsilon} L^2(M)} \leq C h^{-J} \| f \|_{x^{1/2 + \epsilon}L^2(M)}.
 \end{equation}
Let $K > 0$, and suppose that $A \in \Psi^{0, 0, 0}(M)$ satisfies 
$\WF'(A)\cap {^\SC T}_{\pl\bbar{M}}^*\bbar{M}\cap \{\nu=\pm 1,\mu=0\}=\emptyset$, then there exists $C'>0$ such that 
for all $h\in(0,h_0)$ and all $f \in x^{-K} L^2(M)$,
\[ \big\| \big( h^2 \Delta - (1 \mp i\gamma)\big)^{-1} Af \big\|_{ x^{-K-1} L^2(M)} \leq C h^{-J} \| f \|_{x^{-K}L^2(M)}.\]
\end{lemma}

\begin{lemma}\ilabel{lem:poscomm2} 
Let $(M,g)$ be an aymptotically conic manifold and assume an a priori tempered estimate \eqref{assume}
for the resolvent. Let $A_1,A_2\in \Psi^{0,0,0}(M)$ such that:\\
1) $A_1$ is not outgoing-related to $A_2$, \\
2) $\WF'(A_1)$ does not intersect the backward trapped set $\Gamma_-$,\\
3) $\WF'(A_2)\cap {^\SC T}^*_{\pl\bbar{M}}\bbar{M}\cap \{\nu=-1,\mu=0\}=\emptyset$\\
4) $\WF'(A_1)\cap {^\SC T}^*_{\pl\bbar{M}}\bbar{M}\cap \{\nu=1,\mu=0\}=\emptyset$.\\
Then for any $L,K,K'\geq 0$, there is $C>0$ such that for any $f \in x^{-{K'}} L^2(M)$, all $\gamma>0$ small and all $h\in(0,h_0)$
\begin{equation}
  ||A_1(h^2\Delta-(1+i\gamma))^{-1}A_2f ||_{x^KL^2(M)}\leq Ch^{L}||f||_{x^{-K'}L^2(M)}.
\ilabel{A3}\end{equation}
\end{lemma}
Both these lemmas are proved in a very similar manner to the argument in \cite[Section 3]{VaZw}. Lemma \ref{lem:poscomm2} is proved  in 
\cite[Lemma 2]{Da} when $A_1,A_2$ have compact support and when $g$ is non-trapping. Because of this, we will only give the main arguments in the proof and we refer to these articles for details. 

\begin{remark} 
By sacrificing one factor of $h$ in \eqref{A3}, we can change the norm on the left hand side to the $x^K H^1(M)$ norm. 
Then Lemma~\ref{lem:poscomm2} and the Sobolev embedding $H^1\subset L^{p'}$ 
give another proof of Proposition \ref{microlocaltrivial}.
\end{remark}

\begin{proof}[Proof of Lemma~\ref{lem:poscomm1}]
We only prove this lemma for $(h^2\Delta-1 + i\gamma)^{-1}$, the other case is similar. Note that it is sufficient to prove a dual statement: that is, to show that the operator 
\begin{equation}
\text{$A^* (h^2\Delta-1 - i\gamma)^{-1}$ maps $x^{K+1} L^2(M)$ to $h^{-J} x^K L^2(M)$ uniformly in $(h, \gamma)$.} \ilabel{dualstatement}\end{equation}
The rest of the proof is devoted to proving \eqref{dualstatement}. 

We first divide $A$ into an elliptic part and a propagating part. 
Choose a smooth function $\psi$ of a real variable, equal to $1$ in the interval $[1 - \delta/2, 1 + \delta/2]$ and supported in $[1 - \delta, 1 + \delta]$ for sufficiently small $\delta>0$. We decompose $A^* = A^*  (\Id - \psi(P)) + A^*\psi(P)$, and call the first term the elliptic part. As a pseudodifferential operator $A^*$ is uniformly bounded on weighted $L^2$ spaces. Also, using Lemma 2.2 of \cite{VaZw} we see that $(1 - \psi(h^2 \Delta))(h^2 \Delta - (1 + i\gamma))^{-1}$ is bounded uniformly on weighted $L^2$ spaces. This shows that the elliptic part satisfies \eqref{dualstatement}. 

We first note that  it is easy to prove  \eqref{dualstatement}  if the wavefront set of $A$ is disjoint from the characteristic variety of $h^2 \Delta - 1$, just using elliptic estimates. Thus we may assume that $A$ is microsupported near the characteristic variety, given by $\{ |\zeta|_g = 1\}$, or $\{ \nu^2 + |\mu|_h^2 = 1\}$ near the boundary.  The result is also trivial if $A^*$ is microsupported in $\{ x \geq x_0 > 0 \}$. 
 Due  to the assumption that $A^*$ is microsupported away from the outgoing radial set $\{ \nu = 1, \mu = 0 \}$, we see that  it suffices to prove the Lemma under the following two additional hypotheses, which we record for later use:
 \begin{equation} \begin{aligned}
&\bullet \text{$A^*$ is microsupported in $\{ \nu \leq 1 - \delta, x \leq x_0 \}$ for some $\delta > 0$, $x_0 > 0$;} \\
&\bullet A^* (1 - \phi(h^2 \Delta)) = 0. 
\end{aligned}\ilabel{extra}\end{equation}  
where $\phi$ is a function similar to  $\psi$ above, equal to $1$ near $1$  but supported on $[1 - 2\delta, 1+ 2\delta]$. Notice that $\phi(h^2\Delta)\in \Psi^{-\infty,0,0}(M)$
 and has principal symbol $\phi(p)$, see \cite{VaZw,DiSj}.

Under this assumption for the remainder of the proof, we begin by proving \eqref{dualstatement} for $K = -\epsilon$, $\epsilon > 0$ small. So we take $f \in x^{1 - \epsilon} L^2(M)$.  Let $\rho,\til{\rho}\in C_0^\infty(\rr)$ be non increasing, equal to 
$1$ near $0$ and $0$ in $[x_0,\infty)$ for some small $x_0>0$, and with $\rho\til{\rho}=\rho$. 
Let $\chi\in C^\infty(\rr)$ be 
non-increasing, with $\chi'\leq-c_1\chi$, positive in $]-\infty,\nu_0]$ and $0$ in $[\nu_1,\infty)$ 
for some $0<\nu_0<\nu_1<1$ both close to $1$ and some $c_1>0$ large. Let $\til{\phi}\in C_0^\infty(\rr)$
equal $1$ near $1$ and $\phi\til{\phi}=\til{\phi}$. 
It suffices to prove the Lemma for $A^*={\rm Op}_h(a)$ with $a=\til{\rho}(x)\chi(\nu)\til{\phi}(p)$.

Following Section 4 \cite{VaZw},  there exists a real valued symbol such that
\begin{equation}
\text{
$q \in S^{-\infty,-1/2+\eps,0}(M)$ such that $H_p(q^2) = - b \phi(p)^2$ with $b\rho^2(x) \geq c_0 x^{2\epsilon} a^2\rho^2(x)$}
\ilabel{qsymbol}\end{equation}
and such that there exists $C_0>0$ such that
\begin{equation}
C_0 a^2\rho^2(x) \geq x^{1 - 2\epsilon} q^2\rho^2(x).
\ilabel{aqrelation}\end{equation}
It suffices to take $q:=x^{-\demi+\eps}\til{\rho}(x)\til{\phi}(p)\chi(\nu)$ and use the identity for $\alpha\in\rr$
\begin{equation}\ilabel{Hpq}
x^{-\alpha-1}H_p(x^{\alpha}\til{\rho}(x)\chi(\nu))=-2\nu(\alpha \til{\rho}(x)+{\til{\rho}}\,'(x))\chi(\nu)+2|\mu|_{h_0}^2\chi'(\nu)\til{\rho}(x)+\mc{O}(x).\end{equation}

 Let $Q={\rm Op}_h(q)$, and  $B={\rm Op}_h(b)$. It follows that, with $P:=h^2\Delta$, we have
\begin{equation}
i[Q^*Q, P] = h \phi(P) B \phi(P) + h^2 R,
\ilabel{comm-identity}\end{equation}
with $R\in \Psi^{-\infty,1+2\eps,0}(M)$. Let $f\in x^{1-\eps}L^2(M)$ and $u:=(P-(1+i\gamma))^{-1}f$. 
We then follow the argument of \cite[Section 3]{VaZw} to deduce 
$$
\langle u, i[Q^* Q, P] u \rangle = -2 \Im \langle u, Q^* Q (P - (1 + i\gamma)) u \rangle - 2 \gamma \| Qu \|^2. 
$$
It follows from \eqref{comm-identity} that we have 
\begin{equation}
h \langle u, \phi(P) B \phi(P) u \rangle \leq 2 \big| \langle u, Q^* Q (P - (1 + i\gamma)) u \rangle \big|  + h^2 \big| \langle  u,R  u \rangle \big|. 
\ilabel{est1}\end{equation}
We estimate 
\begin{equation}
\big| \langle u, Q^* Q (P - (1 + i\gamma)) u \rangle \big|  \leq  \frac{h}{2C_0} \| x^{1/2} Qu \|_{L^2}^2 + \frac{2C_0}{h} \big\| x^{-1/2} Q (P - (1 + i\gamma)) u \big\|_{L^2}^2. 
\ilabel{est1.5}\end{equation}
Using  that $Q\in \Psi^{-\infty,-\demi+\eps,0}(M)$, we find that (for small $\eps$)
\begin{equation}
\big\| x^{-1/2} Q (P - (1 + i\gamma)) u \big\|_{L^2}^2  \leq  C \| f \|^2_{x^{1- \epsilon}L^2}. 
\ilabel{est2}\end{equation}

Next we use \eqref{qsymbol} and  the sharp Garding inequality \cite[Lemma 2.1]{VaZw}Êto deduce that as operators, 
\begin{equation*}
 \rho(x)\phi(P) B  \phi(P)\rho(x) \geq   \phi(P) A \rho^2(x)x^{2\eps}A^*  \phi(P)  + h R_1,
\end{equation*}
for some pseudo $R_1\in \Psi^{-\infty,1+2\eps,0}(M)$. Since $A^* = A^*  \phi(P)$, we have
 \begin{equation}
\| \rho(x)x^\epsilon A^*u \|^2 \leq  \langle u,  \rho(x)\phi(P) B  \phi(P) \rho(x)u \rangle + h \big| \langle R_1u, u \rangle \big|.
\ilabel{est3}\end{equation}
We also use \eqref{aqrelation} and sharp Garding to deduce 
\begin{equation}
Q^* \rho^2(x)x Q  \leq C_0 A \rho^2(x)x^{2\eps}A^*  + h R_2 ,
\ilabel{sG2}\end{equation}
with $R_2\in \Psi^{-\infty,1+2\eps,0}(M)$, implying that 
\begin{equation}
\| \rho(x)x^{1/2} Qu \|_{L^2}^2 \leq C_0 \| \rho(x)x^\epsilon A^*u \|_{L^2}^2 + h \big| \langle R_2 u, u \rangle \big|.
\ilabel{QArelation}\end{equation}
By \eqref{assume} and the compact support of $(1-\rho(x))$, we also have 
\begin{equation}\ilabel{useassume}
 \begin{gathered}
\|(1-\rho(x))x^{1/2} Qu \|_{L^2} \leq Ch^{-J}\| f\|_{x^{1/2+\eps}L^2},\\ 
|\langle (1-\rho(x))u, \phi(P) B \phi(P) u \rangle|+|\langle u, \phi(P) B \phi(P)(1-\rho(x)) u \rangle|\leq Ch^{-2J}\| f\|^2_{x^{1/2+\eps}L^2}.
\end{gathered}\end{equation}
We combine this with \eqref{est1}, \eqref{est1.5}, \eqref{est2} and \eqref{QArelation} to deduce 
\begin{equation*}
 \langle u,  \phi(P) B  \phi(P) u \rangle \leq C h^{-2J} \| f \|^2_{x^{1- \epsilon}L^2(M)} + \frac{1}{2}\| \rho(x)x^\epsilon A^*u \|^2 + h \big| \langle R_2 u, u \rangle \big|  + h\big| \langle u, R u \rangle \big|. 
\end{equation*}
We insert this in \eqref{est3} and use again \eqref{useassume}Êto find that 
\begin{equation}\begin{gathered}
\| \rho(x)x^\epsilon A^*u \|_{L^2}^2 \leq  C h^{-2J} \| f \|^2_{x^{1- \epsilon}L^2(M)}  
  + Ch (\big| \langle R_1 u, u \rangle \big|+\big| \langle R_2 u, u \rangle \big|  + \big| \langle u, R u \rangle \big|) .
\end{gathered}\ilabel{est4}\end{equation}
The  terms involving $R$ and $R_j$ can be estimated by $C h^{-2J} \| f \|_{x^{1/2 + \epsilon} L^2}$, using \eqref{VZest} again and the fact that 
 $R\in \Psi^{-\infty,1+2\eps,0}(M)$ . Thus this proves the claim since obviously $||(1-\rho)x^{\eps}A^*u||_{L^2}$ is 
estimated by $Ch^{-J}||f||_{x^{1/2+\eps}L^2}$ by \eqref{useassume} and the fact that $(1-\rho(x))$ has compact support. 
 That is, we have shown that 
 \begin{equation}
\| x^\epsilon A^*u \|^2 \leq  C h^{-2J} \| f \|^2_{x^{1- \epsilon}L^2(M)}  ,
\ilabel{est5}\end{equation}
which proves \eqref{dualstatement}  for $K = -\epsilon$, $\epsilon > 0$ small. 

Now to prove for larger values of $K$, we proceed by induction. Given \eqref{dualstatement} for a particular value of $K$, we prove true for $K + l$, where $0 < l \leq 1/2$. To do this, we now let $f \in x^{K+1 + l} L^2(M)$ and as before we find a symbol 
\begin{equation}
\text{
$q \in S^{-\infty,-K-1/2-l,0}(M)$ such that $H_p(q^2) = - b \psi(p)^2$ with $b\rho^2(x) \geq c_0 x^{-2K-2l} a^2\rho^2(x)$.}
\ilabel{qsymbol2}\end{equation}
and for which there is a constant $C_0>0$ with  
\begin{equation}
C_0 a^2\rho^2(x) \geq x^{-2K-1-2l} q^2\rho^2(x).
\ilabel{aqrelation2}\end{equation}
Then we have \eqref{comm-identity} where now $R$ has order $-2K+1-2l \geq -2K$. We follow the line of argument above until \eqref{est2} which is replaced by 
\begin{equation}
\big\| x^{-1/2} Q (P - (1 + i\gamma)) u \big\|^2  \leq  C h^{-1} \| f \|^2_{x^{K+1 + l}L^2(M)} 
\ilabel{est22}\end{equation} 
since now $Q$ has order $-K-1/2 - l$. 

We replace the sharp Garding inequality by 
\begin{equation}
 \rho(x)\phi(P) B  \phi(P)\rho(x) \geq   \phi(P) A \rho^2(x)x^{-2K-2l}A^*  \phi(P)  + h R_1,
 \ilabel{sG3}\end{equation}
and \eqref{sG2} by 
\begin{equation}
Q^* \rho^2(x)x Q  \leq C_0 A \rho^2(x)x^{-2K-2l}A^*  + h R_2 ,
\ilabel{sG4}\end{equation}
where the $R_j\in \Psi^{-\infty,-2K-2l+1}(M)$. Since $-2K+1-2l \geq -2K$, using the inductive assumption 
$u \in h^{-J} x^{-K} L^2$ uniformly shows that the $R$ and $R_j$ terms can be controlled as above. 
 Then the rest of the argument proceeds as above to yield 
 \begin{equation*}
\| x^{-K-l} A^*u \|^2 \leq  C h^{-2J} \| f \|^2_{x^{K+1+l}L^2(M)}  ,
\end{equation*}
which proves \eqref{dualstatement}  for $K+l$. Thus, starting from $K = -\epsilon$, a finite number of iterations gives \eqref{dualstatement} for any $K > 0$. 
\end{proof}

\begin{proof}[Proof of Lemma~\ref{lem:poscomm2}] 
This lemma is proved in a very similar way. Again we easily deal with the case that the microsupport of $A_1$ is disjoint from the characteristic variety, and reduce to the case where $A_1$ satisfies conditions \eqref{extra}. 
 We proceed by induction. Lemma~\ref{lem:poscomm1} gives a starting point for the induction: we can take $L = -J$ and $K= -K'-1$. Now assume that the have the result for some $(K', L)$. We wish to show that the conclusion is valid for $(K+1/2, L+1/2)$. 
Take $f \in x^{-K'} L^2(M)$ and it suffices to assume that $A_i={\rm Op}_h(a_i)$ with $a_1,a_2\in S^{-\infty,0,0}(M)$.

Assume, following \cite{VaZw}, that we have a symbol
\begin{multline}
\text{
$q \in S^{-\infty,-K-1,0}(M)$ such that $H_p(q^2) = - b \phi(p)^2$ with $b \geq c_0 x^{-2K-1} a_1^2$} \\ \text{and $q = 0$ on all points outgoing-related to supp $a_2$.}
\ilabel{qsymbol3}\end{multline}
where $\phi$ is like in the previous Lemma.  The construction of the symbol is explained later.
As a consequence of \eqref{qsymbol3}, we can find quantizations $Q$, $B$ of $q, b$ respectively, such that 
\[
i[Q^*Q, P] = h \phi(P) B \phi(P) + h^2 R,
\]
with $R\in\Psi^{-\infty,-2K,0}(M)$ and such that 
\begin{equation}
\text{the microsupports of $Q$, $B$ and $R$ are not outgoing-related to that of $A_2$.}
\ilabel{Rcond}\end{equation}
 
 We write $v = (P - (1 + i\gamma))^{-1} A_2 f$ and use the identity 
$$
\langle v, i[Q^* Q, P] v \rangle = -2 \Im \langle v, Q^* Q (P - (1 + i\gamma)) v \rangle - 2 \gamma \| Qv \|^2 .
$$
We then apply very similar arguments to those in \cite{VaZw} and the previous lemma to deduce that 
\begin{equation}
h \langle v, \phi(P) B \phi(P) v \rangle \leq 2 \big| \langle v, Q^* Q (P - (1 + i\gamma)) v \rangle \big|  + h^2 \big| \langle v, R v \rangle \big|. 
\ilabel{est11}\end{equation}
The first term is estimated as follows: we write $(P - (1 + i\gamma)) v = A_2 f$ and estimate for $N$ large
$$
\big| \langle v, Q^* Q (P - (1 + i\gamma)) v \rangle \big| \leq \| h^N x^{N} Qv \| \| h^{-N} x^{-N} Q A_2 f \| .
$$
Since $\WF'(Q)\cap\rm{WF}(A_2)=\emptyset$,  $QA_2v$ is $\mc{O}(h^\infty)$ in any weighted space so the second norm is bounded by $C \| f \|_{x^{-K'} L^2}$. The first term is controlled by Lemma~\ref{lem:poscomm1}. Thus we find that 
\begin{equation}
\big| \langle v, Q^* Q (P - (1 + i\gamma)) v \rangle \big| \leq Ch^{2L+1} \| f \|_{x^{-K'} L^2}^2. 
\ilabel{est6}\end{equation}

Next we use \eqref{qsymbol3}, the sharp Garding inequality, and \eqref{extra} to deduce that as operators, 
\begin{equation*}
\phi(P) B \phi(P) \geq A_1^* x^{-2K-1}A_1 + h R_1 
\end{equation*}
for some $R_1\in \Psi^{-\infty,-2k,0}(M)$.  Therefore 
 \[
\| x^{-K-1/2} A_1 v \|^2 \leq  \langle v, \phi(P) B \phi(P) v \rangle + h \big| \langle R_1v, v \rangle \big|.
\]
Combining this with \eqref{est11} and \eqref{est6}, we find that 
\[\| x^{-(K+1/2)} A_1 v \|^2 \leq  Ch^{2L+1} \| f \|_{x^{-K'} L^2}^2 + h \big| \langle v, R v \rangle \big| + h \big| \langle R_1v, v \rangle \big| . 
\]
Let us explain briefly how to deal with the $R$ term (the $R_1$ term is similar): from \eqref{Rcond}
we can rewrite $\cjg v,Rv\cjd=\cjg R_2v,R_3v\cjd$ for some $R_2,R_3\in \Psi^{-\infty,-k,0}(M)$ with microsupport
not outgoing related to $A_2$, and then using the induction assumption, we see that this term is bounded 
 by $Ch^{2L} \| f \|_{x^{-K'} L^2}^2$.
 We conclude that 
\[
\| x^{-(K+1/2)} A_1 v \|_{L^2}^2 \leq  Ch^{2L+1} \| f \|_{x^{-K'} L^2}^2 . 
\]
This completes the inductive step, and thus proves the Lemma.\\ 

Now we discuss the construction of the symbol. It suffices to do this assuming that $A_1$ is microsupported in an arbitrarily small neighbourhood of a point $\xi_1 \in {^\SC T^*}\bbar{M}$. It also suffices to assume that the microsupport of $A_2$ is contained in an arbitrarily small neighbourhood of the characteristic variety $\{ \nu^2 + |\mu|^2 = 1\} $, otherwise we are in the easy elliptic case. 
Using the hypothesis that $\WF'(A_2)$ is disjoint from the incoming radial set, this allows us to assume without loss of generality that 
\begin{equation}
\nu \geq \nu_0 > -1 \text{ on } \WF'(A_2). 
\ilabel{assA2}\end{equation}

First assume that $\xi_1$ is in the incoming radial set $\{ x = 0, \mu = 0, \nu = -1 \}$. Then using \eqref{assA2} we can use a symbol $q$ of the form 
\begin{equation}
x^{-K-1}\til{\rho}(x)\til{\phi}(p)\chi(\nu), \quad \chi(\nu) \text{ supported in } \{ \nu \leq \nu_1 \} \text{ where } -1 < \nu_1 < \min(0, \nu_0) 
\ilabel{qform}\end{equation}
as in the proof of Lemma~\ref{lem:poscomm1}. 

On the other hand, suppose that $\xi_1$ is not in the incoming radial set. By hypothesis is it also not in the outgoing radial set, so the Hamilton vector field $x^{-1} H_p$ is nonzero at $\xi_1$. We may therefore choose coordinates in ${^\SC T^*}\bbar{M}$, valid in a small neighbourhood of the integral curve of $x^{-1} H_p$ through $\xi_1$, in the form $(\Xi, t)$, $\Xi \in \RR^{2n-1}$, $t \in \RR$  such that $\xi_1 = (0,0)$ and $x^{-1} H_p = \partial_t$ in these coordinates, i.e. $\Xi$ is constant along integral curves. We can then find a function of the form $q'_1 = \chi(t) \phi(\Xi)$, where $\chi(t)$ is nonnegative, strictly positive on $(-\infty, t_0)$, $t_0 > 0$ and zero on $[t_0, \infty)$ and $\phi \in C_c^\infty(\RR^{2n-1})$ with $\phi(0) = 1$. Then $x^{-1} H_p(q'_1)$ is nonpositive, and strictly negative at $\xi_1$.  Since $\xi_1$ is not outgoing-related to the microsupport of $A_2$, by choosing $t_0$ small and  the support of $\phi$ sufficiently close to $0$, we can ensure that the support of $q'_1$ is disjoint from $\WF'(A_2)$. Also, since $\xi_1$ is not backward-trapped, if the support of $\phi$ is sufficiently close to $0$ then all integral curves in $\supp \, \phi$ tend to the incoming radial set as $t \to -\infty$. So  we modify $q'_1$ to $q_1 = q'_1 \rho(t _1 - t)$ where $\rho(s)$ is zero for $s \leq 0$, $1$ for $s \geq 1$, and $t_1$ sufficiently negative. Notice that $q_1$ is smooth on ${^\SC T^*}\bbar{M}$, and $x^{-1} H_p q_1$ is strictly negative in a neighbourhood of $\xi_1$ and nonnegative everywhere except on the support of $\rho'(t_1 - t)$ which, for sufficiently negative $t_1$, will be arbitrarily close to the incoming radial set.   Finally we let $q = q_1 + C q_2$ where $q_2$ is of the form \eqref{qform} and $C$ is sufficiently large and this satisfies all conditions.  
\end{proof}

%%%%%%%%%%%%%%%%%%%%%%%%%%%%%%%%%%


\begin{thebibliography}{99}

\bibitem{BSSY} J. Bourgain, P. Shao, C.D. Sogge, X. Yao, \emph{On $L^p$-resolvent estimates and the density of eigenvalues for compact Riemannian manifolds}, arXiv:1204.3927.

\bibitem{ChLiYa} S. Y Cheng, P. Li and S-T. Yau, \emph{On the upper estimate of the heat kernel of a complete Riemannian manifold}, Amer. J. Math \textbf{103} (1981), 1021--1063. 

\bibitem{Da} K. Datchev, \emph{Local smoothing for scattering manifolds with hyperbolic trapped sets.} 
Communications in Mathematical Physics, \textbf{286} (2009), No. 3, pp. 837--850.

\bibitem{DiSj} M. Dimassi, J. Sj\"ostrand, \emph{Spectral asymptotics in the semi-classical limit}, Lecture Note Series 268, London Math Society.

\bibitem{DKS} D. Dos Santos Ferreira, C.E. Kenig, M. Salo, \emph{On $L^p$ resolvent estimates for Laplace-Beltrami operators on compact manifolds}, to appear in Forum Math.

\bibitem{GG} J. Garc\'ia-Cuerva, A. E. Gatto, \emph{Boundedness properties of fractional integral operators associated to non--doubling measures}, Studia Mathematica \textbf{162} (2004), 245--261. 

\bibitem{GH} C. Guillarmou, A. Hassell, \emph{Resolvent at low energy and Riesz transform for Schršdinger operators on asymptotically conic manifolds.} I. Math. Ann. \textbf{341} (2008), no. 4, 859Ð896.

\bibitem{GHS1} C. Guillarmou, A. Hassell, A. Sikora, 
\emph{Resolvent at low energy III: the spectral measure}, Transactions AMS, to appear; arXiv:1009.3084.

\bibitem{GHS2} C. Guillarmou, A. Hassell, A. Sikora, 
\emph{Restriction and spectral multiplier theorems on asymptotically conic manifolds}, arXiv:1012.3780. 

%\bibitem{GCG} J. Garcia-Cuerva, A. E. Gatto, \emph{Boundedness properties of fractional integral operators associated to non-doubling measures}, Studia Math.   

\bibitem{HV} A.~Hassell and A.~Vasy, \emph{The resolvent for Laplace-type operators on asymptotically conic spaces.},  Ann. Inst. Fourier (Grenoble)  \textbf{51}(5) (2001), 1299-1346.


\bibitem{HW} A. Hassell and J. Wunsch, \emph{The semiclassical resolvent and the propagator for non-trapping scattering metrics}, Adv. Math. \textbf{217} (2008), no. 2, 586--682. 

\bibitem{KRS} C. E. Kenig, A. Ruiz and C. D. Sogge, \emph{Uniform Sobolev inequalities and unique continuation for 
second order constant coefficient differential operators}, Duke Math. J. \textbf{55} (1987), 329--347. 

\bibitem{Melb}  R.B. Melrose, \emph{The Atiyah-Patodi-Singer index theorem} 
(AK Peters, Wellesley, 1993).

\bibitem{Mel} R.B. Melrose, \emph{Spectral and scattering theory for the Laplacian on asymptotically Euclidean spaces}. 
Spectral and scattering theory (Sanda, 1992), 85Ð130, Lecture Notes in Pure and Appl. Math., 161, Dekker, New York, 1994.

\bibitem{MZ} R. B. Melrose and M. Zworski, \emph{Scattering metrics and geodesic flow
  at infinity}, Invent. Math. \textbf{124} (1996), no.~1-3, 389--436.

\bibitem{RoTa}    I. Rodnianski, T. Tao, \emph{Effective limiting absorption principles, and applications},
 arXiv:1105.0873. 

\bibitem{Sh} Z. Shen, \emph{On absolute continuity of the periodic Schr\"odinger operators}, 
Internat. Math. Res. Notices \textbf{1} (2001), 1--31.

\bibitem{St} E. Stein, \emph{Harmonic Analysis. Real variable methods, orthogonality and oscillatory integrals.} Princeton University Press, 1993.

\bibitem{Var} N. Varopoulos, \emph{Hardy-Littlewood Sobolev for semigroups}, J. Funct. Anal. \textbf{63} (2) (1985) 240--260

\bibitem{VaZw} A. Vasy, M. Zworski, \emph{Semi-classical estimates in asymptotically Euclidean scattering}
Commun. Math. Phys. \textbf{212} (2000) 205--217. 

\bibitem{WuZw} J. Wunsch, M. Zworski, \emph{Distribution of resonances for asymptotically Euclidean manifolds.} 
J. Differential Geom. \textbf{55} (2000), no. 1, 43Ð82.

\bibitem{Zw} M.~Zworski, \emph{Semiclassical analysis}
to appear in Graduate Studies in Mathematics, AMS, 2012.


\end{thebibliography}
\end{document}